\documentclass[12pt]{article}
\usepackage[numbers,sort&compress]{natbib}
\usepackage{amssymb,amsmath,mathrsfs,amsfonts,amsthm,cases}
\usepackage{enumerate}
\usepackage{xcolor}
\usepackage[colorlinks,
            linkcolor=teal,
            anchorcolor=blue,
            citecolor=purple]{hyperref}
\usepackage{paralist}
\usepackage{color}
\usepackage{epsfig}
\usepackage{graphicx}
\usepackage{setspace}
\usepackage{appendix}
\begin{document}
 \def\pd#1#2{\frac{\partial#1}{\partial#2}}
\def\dfrac{\displaystyle\frac}
\let\oldsection\section
\renewcommand\section{\setcounter{equation}{0}\oldsection}
\renewcommand\thesection{\arabic{section}}
\renewcommand\theequation{\thesection.\arabic{equation}}

\newtheorem{theorem}{Theorem}[section]

\newtheorem{lemma}[theorem]{Lemma}
\newtheorem{prop}[theorem]{Proposition}
\newtheorem*{con}{Conjucture}
\newtheorem*{questionA}{Question}
\newtheorem*{thmA}{Theorem A}
\newtheorem*{thmB}{Theorem B}
\newtheorem{remark}{Remark}[section]
\newtheorem{definition}{Definition}[section]

\title{Global Stability of a PDE-ODE model for acid-mediated tumor invasion
\thanks{
The first author is  supported by NSF of China (No. 11971498,  12126609).
The second author is supported by the National Key R\&D Program of China  (No. 2020YFA0712500) and NSF of China (No. 12126609). } }

\author{Fang Li\\ {School of Mathematics, Sun Yat-Sen University,}\\{\small No. 135, Xingang Xi Road, Guangzhou 510275, P. R. China.}\\
Zheng-an Yao\\{School of Mathematics, Sun Yat-sen
University,}\\{\small No. 135, Xingang Xi Road, Guangzhou 510275, P. R. China.} \\
Ruijia Yu{\thanks{Corresponding author.
E-mail: yurj5@mail2.sysu.edu.cn}}
\\{School of Mathematics, Sun Yat-sen
University,}\\
{\small No. 135, Xingang Xi Road, Guangzhou 510275, P. R. China.} }
\date{}
\maketitle{}

\begin{abstract}
In this paper, we study the global dynamics of a general reaction-diffusion model  based on acid-mediated  invasion hypothesis, which is a candidate explanation for the Warburg effect. A key feature of this model is the density-limited tumor diffusion term for tumor cells, which might give rise to the degeneracy of the parabolic equation. Our theoretical results characterize the effects of  acid resistance and mutual  competition  of healthy cells and tumor cells on tumor progression in the long term,  i.e.,    whether the healthy cells and tumor cells coexist or the tumor cells prevail after tumor invasion. The approach relies on the construction of suitable Lyapunov functionals  and upper/lower solutions. This paper continues and improves the work begun  in \cite{tao}.
\end{abstract}


\section{Introduction}
In this paper, to understand  tumor progression in the long term, we mainly  investigate the global dynamics of a reaction-diffusion model in cancer invasion   proposed by McGillen et al. in \cite{McGillen}
\begin{equation}\label{model}
	\begin{cases}
	u_t=u\left(1-u-a_2 v\right)-d_1uw, & x \in \Omega, t>0, \\
    v_t=D \nabla \cdot((1-u) \nabla v)+r v\left(1-a_1 u-v\right)-d_2vw, & x \in \Omega, t>0, \\
    w_t=\Delta w+c(v-w), & x \in \Omega, t>0, \\ \partial_\nu v=\partial_\nu w=0, & x \in \partial \Omega, t>0, \\
    u(x, 0)=u_0(x), \ v(x, 0)=v_0(x), \ w(x, 0)=w_0(x), & x \in \Omega,
    \end{cases}
\end{equation}
where $\Omega$ is a smooth and bounded domain in $\mathbb R^n$, $\nu$ denotes the unit outward normal vector  on  $\partial \Omega$, and  $u,v,w$ represent the density functions of healthy cells, tumor cells and lactic acid in tissue microenvironment  respectively. Also,  $D$, $d_1$, $d_2$, $r$, $c$, $a_1$, $a_2$ are positive non-dimensional parameters, where $D$ is the diffusion rate of tumor cells, $d_1$ and $d_2$ are the death rates of healthy cells and tumor cells caused by the lactic acid, respectively, $a_1$ and $a_2$ represent the competition coefficients.

The formulation of the model (\ref{model})  is based on acid-mediated  invasion hypothesis \cite{garenby}, which is a candidate explanation for the Warburg effect \cite{warburg}, a widespread preference in tumors for cytosolic glycolysis rather than oxidative phosphorylation for glucose breakdown.
Altered glucose metabolism in tumor cells plays a critical role in cancer biology, through this process, the tumor cells change the microenvironment by producing acid \cite{gatenby1995}. Microenvironemtal acidosis is toxic to normal cells since it could lead to cellular necrosis and apoptosis\cite{park}\cite{williams}.
The acid-mediated invasion hypothesis is motivated by viewing the tumor as an invasive species, which gains powerful selective advantage by producing lactic acid into the microenvironment since tumor cells acquire resistance to acidification of microenvironment while healthy cells are acid-sensitive \cite{garenby}.

Gatenby and Gawlinski were the first to put this hypothesis into a reaction-diffusion framework \cite{gatenby1996} as follows
 $$
\begin{cases}
	u_t=u\left(1-u\right)-d_1u w, & x \in \Omega, t>0, \\
    v_t=D \nabla \cdot((1-u) \nabla v)+r v\left(1-v\right), & x \in \Omega, t>0, \\
    w_t=\Delta w+c(v-w), & x \in \Omega, t>0 \\ \partial_\nu v=\partial_\nu w=0, & x \in \partial \Omega, t>0, \\
    u(x, 0)=u_0(x), \ v(x, 0)=v_0(x), \ w(x, 0)=w_0(x), & x \in \Omega.
    \end{cases}
$$
A key feature of the Gatenby-Gawlinski model is the density-limited tumor diffusion term in the second equation, which indicates that the tumor will be spatially constrained when the  healthy cells reach the carrying capacity. The studies of the Gatenby-Gawlinski model suggest that  acidity may play an important role in tumor progression \cite{gillies}. In 2006, Gatenby et al. further their work by generalizing the Gatenby-Gawlinski model and comparing the numerical results with the experimental results. In mathematic point of view, their work confirms that the acid-mediated tumor invasion model could make detailed predictions\cite{gatenby2006} .

Later, in order to further the understanding  of acid-mediated invasion and capture a wider range of tumor behaviors which may be
clinically relevant, the model (\ref{model}) is proposed on the basis of the Gatenby-Gawlinski model   by incorporating terms representing the mutual  competition between healthy cells and tumor cells, and the acid-mediated tumor cell death.  In \cite{McGillen}, the invasive behaviors of tumor cells are  characterized by  numerical methods and an asymptotic traveling wave analysis.

Among other things, the
linear stability of the steady states in the model (\ref{model}), which reflects the  invasive and non-invasive behaviors of tumors,  is
analyzed in \cite{McGillen} as follows:
\begin{itemize}
\item a trivial absence of all species, $(u, v, w) = (0, 0, 0)$, globally unstable;
\item a healthy state, $(u,v,w)=(1,0,0)$  linearly unstable if $a_1<1$ and linearly stable if $a_1>1$;
\item  a homogeneous tumor state,
      $$\displaystyle (u,v,w)=
      \left( 0,\left(1+\dfrac{d_2}{r}\right)^{-1},\left(1+\dfrac{d_2}{r}\right)^{-1} \right),$$
    linearly unstable if $\frac{d_2}{r}>a_2+d_1-1$ and linearly stable if $\frac{d_2}{r}<a_2+d_1-1$;
\item a heterogeneous state, $(u,v,w)=(u^*,v^*, w^*)$, where $u^*>0, v^*>0, w^*>0$. Direct computation yields that $(u^*,v^*, w^*)=(1-(a_2+d_1)v_h,v_h,v_h)$, where
    $$
     v_h:=\dfrac{1-a_1}{1-a_1a_2+\dfrac{d_2}{ r} -a_1d_1}, \quad a_1\neq 0.
    $$
    Moreover,
    \begin{itemize}
    \item  if $a_1>1$, then     $u^*,\, v^*,\, w^*$ are positive if and only if  $\frac{d_2}{r}<a_2+d_1-1$.  Also $(u^*,v^*, w^*)$ is linearly unstable;
    \item if $a_1<1$, then $u^*,\, v^*,\, w^*$ are positive if and only if  $\frac{d_2}{r}>a_2+d_1-1$. Moreover,  $(u^*,v^*, w^*)$ is linearly stable.
    \end{itemize}
    \end{itemize}
Based on the  analysis of  linear stability,
when $a_1<1$,  the healthy   state   is locally unstable. Thus tumor invasion could happen, and naturally the heterogeneous state and the homogeneous tumor state are two possible outcomes of the invasive behaviors of tumors.  
To further understand and characterise  tumor progression in the long term,  i.e.,    whether the healthy and tumor cells coexist or the tumor cells prevail after tumor invasion,   we analyze the global dynamics of  the model (\ref{model}).  This issue  is   much more complicated and far from being understood.
In this paper,  we focus on studying the global stability of the three nontrivial steady states:
\begin{itemize}
 \item  the healthy state $(1,0,0)$,
 \item  the homogeneous tumor state  $ \displaystyle\left( 0,\left(1+\dfrac{d_2}{r}\right)^{-1},\left(1+\dfrac{d_2}{r}\right)^{-1} \right)$,
 \item  the heterogeneous state $(u^*,v^*, w^*)$ ,
\end{itemize}
 when they exist and  are locally stable and manage to characterize the ranges of the parameters in the model (\ref{model}) where {\it the local stability implies the global stability.}


\medskip
For clarity, throughout this paper, we always assume that
 the initial data   $u_0,v_0, w_0$ satisfy the condition
    \begin{equation}\label{ini}
	\left\{
    \begin{aligned}
	&    u_0\in W^{2,\infty}(\Omega),\ 0<u_0<1\ \textrm{in}\ \bar{\Omega},\\
	 &   v_0\in W^{2,\infty}(\Omega),\ v_0>0\ \textrm{in}\ \bar{\Omega},\\
	 &   w_0\in W^{2,\infty}(\Omega),\ w_0>0\ \textrm{in} \ \bar{\Omega}.\\
    \end{aligned}
    \right.
    \end{equation}
In \cite{tao}, the global stability of the heterogeneous state $(u^*,v^*, w^*)$ is studied and partial results are obtained.
\begin{thmA}[\cite{tao}]\label{tao}
    	Let $a_1>0$, $a_2>0$, $d_1\geq 0$, $d_2\geq 0$, $r>0$, $D>0$, $c>0$ satisfy
    \begin{equation}\label{basic}
   a_1<1,\	\dfrac{d_2}{r}>a_2+d_1-1\ \textrm{and}
    \end{equation}
    \begin{equation}\label{worse_d_2_2}
    	\dfrac{d_2}{r}<1-a_1a_2-a_1d_1,
    \end{equation}
    then the solution $(u,v,w)$ to the problem $(\ref{model})$  exists globally in time and enjoys the property that
    $$ ||u(\cdot,t)-u^*||_{L^\infty(\Omega)}+||v(\cdot,t)-v^\ast||_{L^\infty(\Omega)}
    +||w(\cdot,t)-w^\ast||_{L^\infty(\Omega)}\rightarrow 0\quad \text{exponentially  as} \quad t\rightarrow\infty.
	$$
    \end{thmA}


\medskip

Our first main result greatly improve Theorem A regarding the global stability of the heterogeneous state $(u^*,v^*, w^*)$.
    \begin{theorem}\label{thm1}
    In the system $(\ref{model})$, assume that the non-dimensional parameters $D$, $d_1$, $d_2$, $r$, $c$, $a_1$, $a_2$ are positive, the initial data  $(u_0,v_0,w_0)$ satisfies $(\ref{ini})$,  $a_1<1$  and	$a_1a_2<1$. Also
    	if one of the following assumptions holds:
    	\begin{enumerate}[(i)]
    		\item
    		\begin{equation}\label{d2h1}
    			d_1\leq d^h_1 \quad\text{and}\quad
    			\dfrac{d_2}{r}>d_1+a_2-1,
    		\end{equation}
    		\item
    		\begin{equation}\label{d2h2}
    			d_1 > d^h_1 \quad\text{and}\quad
    			\dfrac{d_2}{r}> d^h_2,
    		\end{equation}
    	\end{enumerate}
    	where
    	\begin{equation}\label{d1h}
    		d^h_1=\left(\dfrac{1+\sqrt{1-a_1a_2}}{1-\sqrt{1-a_1}}\right)^2-a_2,
    	\end{equation}
    	
    	\begin{equation}\label{d^h_2}
    		d^h_2=\dfrac{1}{4}\left(
    \frac{a_1(a_2+d_1)}{1+\sqrt{1-a_1a_2}}+1+\sqrt{1-a_1a_2}
    	\right)^2-1,
    	\end{equation}
    	then the solution $(u,v,w)$ to the system $(\ref{model})$  exists globally and enjoys the property that
$$  ||u(\cdot,t)-u^*||_{L^\infty(\Omega)}+||v(\cdot,t)-v^\ast||_{L^\infty(\Omega)}
    +||w(\cdot,t)-w^\ast||_{L^\infty(\Omega)}\rightarrow 0\quad \text{exponentially  as} \quad t\rightarrow\infty.
    $$
    \end{theorem}

As discussed earlier,  the heterogeneous state $(u^*,v^*, w^*)$ with $u^*>0,\, v^*>0,\, w^*>0$ exists and is linearly stable if and only if
$$
a_1<1,\	\dfrac{d_2}{r}>a_2+d_1-1.
$$
Hence  the case (i) in Theorem \ref{thm1} indicates that  if
$$
a_1<1,	\ a_1a_2<1,\ d_1\leq d_1^h,
$$
then the local stability of the heterogeneous state implies its global stability.
Moreover,   in Theorem A,
to guarantee that there exists    $d_2>0$ such that  the conditions (\ref{basic}) and (\ref{worse_d_2_2}) are fulfilled, i.e.,
$$
a_2+d_1-1< \dfrac{d_2}{r}<1-a_1a_2-a_1d_1,
$$
the conditions
$$
a_1a_2<1,\ d_1<\dfrac{2}{1+a_1}-a_2
$$
are necessary. Then  elementary computations yield that when $a_1<1$,
$$
d_1^h=\left(\dfrac{(1+\sqrt{1-a_1a_2})(1+\sqrt{1-a_1})}{a_1}\right)^2-a_2
> \left(\dfrac{1}{a_1}\right)^2-a_2>\frac{2}{1+a_1}-a_2.
$$
The condition (\ref{worse_d_2_2})  in Theorem A also  imposes an upper bound for $d_2$, while this   is not required in (\ref{d2h1}).
Hence the case (i) in Theorem \ref{thm1} improves Theorem A.

Furthermore, the case (ii) that $d_1 > d_1^h$ in  Theorem \ref{thm1} demonstrates that the global stability of the heterogeneous state is still valid when both $d_1$ and $d_2$ are relatively large, i.e., $d_1 > d^h_1$ and	$\frac{d_2}{r}> d^h_2$.  This range is not mentioned in Theorem A.
We also point out that  in the case (ii), the global stability of the heterogeneous state in the range
$$
a_1<1,	\ a_1a_2<1,\ d_1 > d^h_1,\ d_1+a_2 -1 < \dfrac{d_2}{r} \leq d^h_2
$$
is still not clear.

Biologically, Theorem \ref{thm1} demonstrates that the introduction of acid-mediated tumor cell death could result in the coexistence of tumor and healthy cells as long as the healthy cells are not very aggressive in consuming resources, i.e., $a_1<1$ and  the tumor cells are quite sensitive to acid, i.e., $d_2$ is relatively large compared with $d_1$.

\medskip
Our second main result is about the global stability of the homogeneous tumor state, denoted by
$$(0,\tilde v, \tilde w) = \left( 0,\left(1+\dfrac{d_2}{r}\right)^{-1},\left(1+\dfrac{d_2}{r}\right)^{-1} \right).
$$
    \begin{theorem}\label{thm2}
    	In the system $(\ref{model})$, assume that the non-dimensional parameters $D$, $d_1$, $d_2$, $r$, $c$, $a_1$, $a_2$ are positive, the initial data  $(u_0,v_0,w_0)$ satisfies $(\ref{ini})$ and  $a_1<1$.
    	Moreover, if one of the following situations holds:
    	\begin{enumerate}[(i)]
    		\item \begin{equation}\label{thm21}
    		d_1\leq d^c_1\quad
    		           \text{and}\quad
    		          \dfrac{d_2}{r}<\dfrac{d_1+a_2}{\max \{a_1a_2,1\}} -1,
    	          \end{equation}
    	    \item \begin{equation}\label{thm22}
    	    	a_1a_2<1,\quad
    	    	d^c_1<d_1\leq d^h_1\quad
    	    	\text{and}\quad
    	    	\dfrac{d_2}{r}<d_1+a_2-1,
    	    \end{equation}
    	    \item \begin{equation}\label{thm23}
    	    	a_1a_2<1,\quad
    	    	d_1>d^h_1\quad
    	    	\text{and}\quad
    	    	\dfrac{d_2}{r} < d^c_2,
    	    	\end{equation}
    	    	\item \begin{equation}\label{thm24}
    	    	a_1a_2 \geq 1,\quad
    	    	d_1>d^c_1 \quad
    	    	\text{and}\quad
    	    	\dfrac{d_2}{r} < d^c_2,
    	    \end{equation}
    	\end{enumerate}
    	where $d^h_1$ is defined in (\ref{d1h}),
    	\begin{equation}\label{d^c_1}
    		d^c_1=\dfrac
    	{a_1a_2}
    	{(1-\sqrt{1-a_1})^2}-a_2,
    	\end{equation}
    	and
    	\begin{equation}\label{d^c_2}
    		d^c_2=4
    		\left (1-\sqrt{1-a_1} +\dfrac{1}{1-\sqrt{1-a_1}}\dfrac{a_1a_2}{a_2+d_1}
    		\right)^{-2} -1,
    	\end{equation}
    	then the solution $(u,v,w)$ to the system $(\ref{model})$  exists globally  and satisfies the property that
	$$  ||u(\cdot,t)||_{L^\infty(\Omega)}+||v(\cdot,t)-\tilde v||_{L^\infty(\Omega)}
    +||w(\cdot,t)-\tilde w||_{L^\infty(\Omega)}\rightarrow 0\quad \text{exponentially  as} \quad t\rightarrow\infty.
    $$
    \end{theorem}

Recall that  the homogeneous tumor state is linearly stable if $\frac{d_2}{r}<d_1+a_2-1$.  Also notice that when $a_1<1$, $a_1a_2=1$, $d_1^c$ coincides with $d_1^h$. Thus according to the cases (i) and (ii) in Theorem \ref{thm2}, when
$$a_1<1,\ a_1a_2\leq 1,\ d_1\leq d^h_1,$$
 the local stability of the homogeneous tumor state implies the global stability.  While the global stability   in the range
 $$
 a_1<1,	\ a_1a_2\leq 1,\ d_1 > d^h_1,\  d_2^c\leq  \dfrac{d_2}{r} < d_1+a_2 -1
 $$
 is still unknown.

Moreover, when
$$a_1<1,\ a_1a_2 > 1,$$
partial results concerning the global stability of the homogeneous tumor state  are also obtained in Theorem \ref{thm2}. Simply speaking, in this situation, when $d_2$ is relatively small, the homogeneous tumor state is globally stable.

From the view of biology, Theorem \ref{thm2} demonstrates  that as long as the tumor has strong resistance to acid, i.e., $d_2$ is small enough, and the healthy cells are more sensitive to the acid, then the tumor cells can invade the healthy state and eliminate  the healthy cells completely regardless of the competition coefficient $a_2$ for tumor cells.  This confirms that the acid resistance abilities of healthy and tumor cells  play a   significant   role in tumor progression.

In summary, when $a_1<1$, the healthy state can be invaded by tumor cells,  and then Theorems \ref{thm1} and \ref{thm2} together characterize  how   acid resistance and mutual competition of tumor and healthy cells  determine whether tumor cells dominate or two types of cells coexist at reduced densities in the long term.
It is also worth pointing out that if
$$
a_1>1, \, \frac{d_2}{r}<d_1+a_2-1,
$$
then both the healthy state and the homogeneous tumor state are locally stable. Biologically, this is related to the situation that the healthy cells are aggressive in  competition and  the tumor cells have strong acid resistance ability, and thus the tumor progression is expected to be delicate.

\medskip
Our last main result is about the global stability of the healthy state $(1,0,0)$.

    \begin{theorem}\label{thm3}
	In the system $(\ref{model})$, assume that the non-dimensional parameters $D$, $d_1$, $d_2$, $r$, $c$, $a_1$, $a_2$ are positive, the initial data  $(u_0,v_0,w_0)$ satisfies $(\ref{ini})$,  $v_0\leq 1$,  and  $a_1>1$, $a_2<1$.
    	Moreover, if one of the following situations holds:
    	\begin{enumerate}[i)]
    		\item \begin{equation}\label{d2_r_1}
    		           d_1\leq d^r_1\quad \text{and}\quad
    		          \dfrac{d_2}{r}>a_1(d_1+a_2-1),
    	          \end{equation}
    	    \item \begin{equation}\label{d2_r_2}
    	    	d_1>d^r_1\quad \text{and}\quad
    	    	\dfrac{d_2}{r} > d^r_2,
    	    \end{equation}
    	\end{enumerate}
    	where
    		\begin{equation}\label{d1r}
    			d^r_1=(1+\sqrt{1-a_2})^2-a_2,
    		\end{equation}
    		\begin{equation}\label{d2r}
    			d^r_2=\dfrac{a_1}{4}\left(
    \frac{a_2+d_1}{1+\sqrt{1-a_2}}+1+\sqrt{1-a_2}
    	\right)^2-a_1,
    		\end{equation}
    	then the solution $(u,v,w)$ of the system (\ref{model})  exists globally   and  enjoys the property that
	$$  ||u(\cdot,t)-1||_{L^\infty(\Omega)}+||v(\cdot,t)||_{L^\infty(\Omega)}
+||w(\cdot,t)||_{L^\infty(\Omega)}\rightarrow 0\quad \text{exponentially  as} \quad t\rightarrow\infty.
$$
    \end{theorem}

Based on the analysis of local stability, when $a_1>1$,  the healthy state is stable and can prevent the invasion of small amounts of tumor cells. According to Theorem \ref{thm3},
to guarantee the global stability of the healthy state, additionally, we need require that $a_2<1$ and $d_2$ is relatively large compared with $d_1$. This indicates that if both the competition ability and acid resistance of tumors cells are weak,  they will not have the chance to survive in the normal tissue. However, this is not biologically realistic, since usually it is the tumor who exhibits strong acid resistance.  Thus Theorem \ref{thm3} reflects that the acid-mediated invasion is quite effective from the reverse side.

\medskip

At the end, some remarks  from the view of mathematics and strategies of our proofs are appropriate.
\begin{enumerate}
	\item   The system (\ref{model}) can be viewed as a generalization of a classical mutual competition system \cite{Murray}.
	On the basis of Theorems \ref{thm1}, \ref{thm2} and \ref{thm3},
	we observe that the multiplication of competition coefficients  $a_1a_2$ is a critical quantity in determining the global dynamics. To be more specific,  thanks to the case (i) in Theorem \ref{thm1} and the cases (i) and (ii) in Theorem \ref{thm2}, {\it the property that local stability implies global stability} is verified under the conditions that
	$$
	a_1<1, \ a_1a_2<1,\ d_1\leq d^h_1=\left(\dfrac{1+\sqrt{1-a_1a_2}}{1-\sqrt{1-a_1}}\right)^2-a_2.
	$$
	However, when
	$a_1<1,\ a_1a_2 > 1,$
	for any $d_1>0$,
	there always exist the ranges where the homogeneous tumor state is locally stable but the global stability is unknown. Moreover, in Theorem \ref{thm3},  the conditions on $a_1$, $a_2$ automatically exclude the strong competition, i.e., $a_1>1$, $a_2>1$.
	Indeed,   in the studies of the classical Lotka-Volterra system for two competing species,  the multiplication of competition coefficients, still denoted by $a_1a_2$,  also plays an important role in determining  the global dynamics.  It is known that  when $a_1a_2  < 1$, {\it the property that local stability implies global stability} is completely verified, while when $a_1a_2>1$,  the global dynamics becomes very complicated and is far from being understood. See \cite{b_l_cvpde, Iida} and the references therein for more details. Therefore, back to the system (\ref{model}),
	it is naturally expected that more complicated phenomena might happen in the unknown ranges.

	\item The model (\ref{model})  is a combination of ODE and PDEs. A	 novel feature of this model is the density-limited tumor diffusion term in the second equation, which might give rise to the degeneracy of the parabolic equation. In \cite{tao}, where the global stability of the heterogenous state $(u^*,v^*,w^*)$ is studied,   to exclude this possibility of degeneracy, the rectangle method idea is employed to the system  where spatially homogeneous upper and lower solutions are constructed. The conditions on the parameters in Theorem A are imposed to guarantee that the upper solution approaches the lower solution as time goes to infinity.
	
	\item The proofs of our theorems make crucial use of Lyapunov functionals.  Inspired by \cite{bai, Hsu}, we design   proper  Lyapunov functionals for each nontrivial steady state to derive  $L^2$ convergence of $w$, where the conditions on the parameters   are required to warrants the existence of the desired Lyapunov functionals.  Moreover, based on the equation satisfied by $w$,
	the  $L^2$ convergence of $w$ is improved to $L^\infty$ convergence.  Indeed, the $L^2$ convergence of $u,\, v$ are derived at the same time. However, since there is no diffusion term for $u$ and the density-limited tumor diffusion term for $v$ might cause the degeneracy of the parabolic equation, the norm of the convergence cannot be improved for $u,\, v$.
	
	\item   For the heterogenous state $(u^*,v^*,w^*)$, to prove the $L^\infty$ convergence of $u,\, v$,  similar to \cite{tao}, we also employ  the rectangle method idea. However, different from \cite{tao}, where   the auxiliary system of ODEs is constructed for the whole system (\ref{model}) to obtain the upper and lower solutions, we only construct the auxiliary system of ODEs for the first two equations in the system (\ref{model}), where $w$ is replaced by a suitable perturbation of $w^*$. Obviously, the $L^\infty$ convergence of $w$ to $w^*$ guarantees that this simplification is practicable. Thanks to this simplification, no more requirements on the parameters are needed.
\item  For the homogeneous tumor state  and the healthy state, the auxiliary systems of ODEs  are constructed similarly, but the arguments in deriving the convergence relations between lower and upper solutions are different, since in each case, the corresponding elements of the lower solutions are reduced to be zero and the arguments for the heterogenous state  is not applicable anymore.
	
\end{enumerate}

This paper is organized as follows.
Some preliminary results are prepared in Section 2. Sections 3, 4  and 5 are devoted to Theorems \ref{thm1}, \ref{thm2} and \ref{thm3} respectively.


    \section{Preliminary results}


The local existence of solutions to the system (\ref{model}), as well as an extensibility criterion was proved in \cite{tao} as follows.

    \begin{lemma}\label{local_existence}
    	Let $a_1>0, a_2>0, d_1 \geq 0, d_2 \geq 0, r>0, D>0$ and $c>0$, and let $(\ref{ini})$ hold. Then there exist $T_{\max } \in(0, \infty]$ and a unique triple
$$
(u, v, w) \in C^{1,1}\left(\bar{\Omega} \times\left[0, T_{\max }\right)\right) \times\left(C^0\left(\bar{\Omega} \times\left[0, T_{\max }\right)\right) \cap C^{2,1}\left(\bar{\Omega} \times\left(0, T_{\max }\right)\right)\right)^2
$$
solving $(\ref{model})$ classically in $\Omega \times\left(0, T_{\max }\right)$. These functions have the properties
\begin{align}
&0<u<1 & \text { in } \Omega \times\left(0, T_{\max }\right), \label{u_upperbound}\\
&0<v \leq \max \left\{1,\left\|v_0\right\|_{L^{\infty}(\Omega)}\right\} & \text { in } \Omega \times\left(0, T_{\max }\right),\label{v_upperbound} \\
&0<w \leq \max \left\{1,\left\|v_0\right\|_{L^{\infty}(\Omega)},\left\|w_0\right\|_{L^{\infty}(\Omega)}\right\} & \text { in } \Omega \times\left(0, T_{\max }\right).\label{w_upperbound}
\end{align}

Moreover, we have the following dichotomy:
$$
\text {either} \quad T_{\max }=\infty, \quad \text {or} \quad \limsup _{t \nearrow T_{\max }}\|u(\cdot, t)\|_{L^{\infty}(\Omega)}=1.
$$
    \end{lemma}

 Thanks to Lemma \ref{local_existence}, the  the global existence of the solution to the system  (\ref{model})  follows easily.
    \begin{lemma}\label{global_existence}
    	Suppose that assumptions of Lemma $\ref{local_existence}$ hold, then $T_{\max}=+\infty$, namely, (\ref{model}) has a unique global classical solution.
    \end{lemma}

    \begin{proof}
    	To verify the solution is global, we denote $\bar u$ the solution of the following ODE
    $$
    \begin{cases}
    		\dfrac{\mathrm d}{\mathrm dt}\bar u=\bar u(1-\bar u),\quad t>0,\\
    		\bar u(0)=\max\limits_{x\in\bar\Omega}u_0.
    	\end{cases}
    $$
    By comparison principle, $\bar u(t)\geq u(t)$ for all $t\geq 0$. Since $\bar u(0)<1$, it is straightforward to verify that $\bar u$ will not reach 1 in finite time, and the dichotomy in Lemma $\ref{local_existence}$ immediately indicates that $T_{\max}=+\infty$.
    \end{proof}

    The following property is based on elementary analysis  and useful in the proofs of Lemma $\ref{l2con}$, Lemma $\ref{l2contumor}$ and Lemma $\ref{l2con_r}$.

    \begin{lemma}\label{bdd_convergence}
    	Suppose that $f(t)$ is a uniformly continuous nonnegative function defined on $(0,+\infty)$ such that $\displaystyle \int_0^\infty f(t) \mathrm dt<+\infty$, then $f(t)\rightarrow 0$ as  $t\rightarrow \infty$.
    \end{lemma}

    At the end, we show  how to improve the $L^2$ convergence of $w$ to $L^\infty$ convergence. Since the arguments for  all the three nontrivial steady states are the same, we leave it in this  section for simplicity.
    \begin{lemma}\label{wdecay_h}
    	Suppose that $(u,v,w)$ is a global solution of $(\ref{model})$-$(\ref{ini})$ and satisfies
    	\begin{equation}\label{w_L2_temp}
    		||w-\mathfrak w||_{L^2(\Omega)}\rightarrow 0 \quad as\quad t\rightarrow\infty,
    	\end{equation}
        where $\mathfrak w\in\{w^*, \tilde w,0\}$, then
    	$$
    	     ||w-\mathfrak w||_{L^\infty(\Omega)}\rightarrow 0 \quad as\quad t\rightarrow\infty.
        $$
    \end{lemma}

    \begin{proof}
    	Recall from Lemma $\ref{local_existence}$ that $v$ is uniformly bounded, by the smoothing property of $(e^{t(\Delta-c)})_{t>0}$ \cite{winkler2022}, for all $t>0$, there exists a constant $c_n>0$ such that
    	 \begin{eqnarray}\label{wbound}
    	 	||w(\cdot,t)||_{W^{1,\infty}(\Omega)}
    	 & \leq & c_n ||w(\cdot,0)||_{W                                                                             ^{1,\infty}(\Omega)}\cr
    	&& +c_n\Big\{\sup\limits_{\tau>0}||v(\cdot,\tau)||^{(2n-1)/(2n)}_{L^\infty(\Omega)}\Big\}
    	 \Big\{\sup\limits_{\tau>0}||v(\cdot,\tau)||^{1/(2n)}_{L^1(\Omega)}\Big\}  < \infty.
    	 \end{eqnarray}
    	Using Gagliardo-Nierenberg inequality \cite{gn}, we have
    	\begin{equation}\label{gn}
    		||w-\mathfrak w||_{W^{\frac{n}{n+1},2(n+1)}(\Omega)}\leq
    	||w-\mathfrak w||^{{n}/{(n+1)}}_{W^{1,\infty}(\Omega)}||w-\mathfrak w||^{{1}/{(n+1)}}_{L^2(\Omega)}.
    	\end{equation}
    	By fractional Sobolev imbedding \cite{fractional}, there exist constants $0<\beta_n<1$ and $\mathcal C_n>0 $ such that
    	\begin{equation}\label{sobolev}
    		||w-\mathfrak w||_{ C^{\beta_n}(\bar \Omega)}\leq
    	\mathcal C_n||w-\mathfrak w||_{W^{\frac{n}{n+1},2(n+1)}(\Omega)}.
    	\end{equation}
    	The desired conclusion follows immediately from
    	 $(\ref{w_L2_temp})$ and $(\ref{wbound})$-$(\ref{sobolev})$ .
    \end{proof}

\section{The heterogeneous state }
    This section is devoted to the proof of Theorem $\ref{thm1}$, which is about the global convergence of the  heterogeneous state
     $$
    (u^*,v^*, w^*)=(1-(a_2+d_1)v_h,v_h,v_h),\ \textrm{where}\ v_h:=\dfrac{1-a_1}{1-a_1a_2+\dfrac{d_2}{r} -a_1d_1}, \ a_1\neq 0,
    $$
    where $u^*>0,\, v^*>0,\, w^*>0.$
    Recall that     the heterogeneous state $(u^*,v^*, w^*)$ exists and is linearly stable if and only if
    $$
    a_1<1,\	\dfrac{d_2}{r}>a_2+d_1-1.
    $$
    According to the strategies explained at the end of the introduction, we present the proof in three steps:
    \begin{itemize}
    	\item in Section \ref{l2heter},  we   demonstrate the $L^\infty$ convergence of $w$ to $w^*$ with the help of  a Lyapunov functional, the form of which is inspired by \cite{bai, Hsu}.
    	\item in Section \ref{auxh},  the auxiliary system of ODEs for the first two equations in the system (\ref{model}) is constructed and some properties are prepared.
    	\item  in Section \ref{ifih}, the $L^\infty$ convergence of $u,\, v$ to $u^*,\, v^*$ respectively is established.
    \end{itemize}

    \subsection{$L^\infty$ convergence of $w$ in the heterogeneous state}\label{l2heter}

  To prove the $L^\infty$ convergence of $w$,  the key step is to select the proper Lyapunov functional  in the following lemma.

    \begin{lemma}\label{uvwh}
    	Suppose that $(u,v,w)$ is the global solution of $(\ref{model})$-$(\ref{ini})$ and the assumptions of Theorem $\ref{thm1}$ hold. Define
    	$$
    	A_h(t)=\int_\Omega u(x,t)-u^*-u^*\ln\dfrac{u(x,t)}{u^*}\,\mathrm dx,
    	$$
    	$$
    	B_h(t)=\int_\Omega v(x,t)-v^*-v^*\ln\dfrac{v(x,t)}{v^*}\,\mathrm dx,
    	$$
    	$$
    	C_h(t)=\dfrac{1}{2}\int_\Omega(w(x,t)-w^*)^2\,\mathrm dx.
    	$$
    	Then there exist $\beta_h>0$, $\eta_h>0$ and $\varepsilon_h>0 $ such that the functions $E_h(t)$ and $F_h(t)$ defined by
    	\begin{equation}\label{eh}
    	E_h(t)=A_h(t)+\dfrac{\beta_h}{r}B_h(t)+\dfrac{\eta_h}{c}C_h(t), \quad t>0
    	\end{equation}
    	and
    	\begin{align}\label{fh}
    	F_h(t)=&
    	\int_\Omega(u(x,t)-u^*)^2\mathrm dx
    	+\int_\Omega(v(x,t)-v^*)^2\mathrm dx
    	+\int_\Omega(w(x,t)-w^*)^2\mathrm dx\nonumber\\
    	&+\int_\Omega \left|\nabla w(x,t)\right|^2\mathrm dx, \quad t>0,
    	\end{align}
    	satisfy
    	$$
    	E_h(t)\geq 0, \quad t > 0
    	$$
    	as well as
    	\begin{equation}\label{decay1}
    	\dfrac{\mathrm d}{{\mathrm dt}}E_h(t)\leq-\varepsilon_h F_h(t).
    	\end{equation}
    \end{lemma}

On the basis of  Lemma \ref{uvwh} is valid, we establish the $L^\infty$ convergence of $w$.

    \begin{lemma}\label{l2con}
    	Suppose that assumptions of Theorem $\ref{thm1}$ hold, and $(u,v,w)$ is the global solution of $(\ref{model})$-$(\ref{ini})$, then
    	$$
    	||w-w^*||_{L^\infty(\Omega)}\rightarrow 0\quad as\quad t\rightarrow\infty.
    	$$
    \end{lemma}

    \begin{proof}
    	Integrating $(\ref{decay1})$ from $0$ to $+\infty$, by the fact that $E^h(\cdot)$ is nonnegative, we have
    	$$
    \int_0^{+\infty}||u-u^*||^2_{L^2}+||v-v^*||^2_{L^2}+||w-w^*||^2_{L^2}+|| \nabla w||^2 _{L^2} \mathrm dt	\leq \dfrac{1}{\varepsilon_h} E^h(0)<+\infty.
    	$$
    	Using standard $L^p$ estimate and Sobolev embedding on $w$, we obtain that there exist two constants $\mathcal C_1>0$ and $0<\alpha_n<1$ such that
    	$$
    	||w||_{C^{\alpha_n,\alpha_n/2}(\bar\Omega\times(k,k+1])}
    	\leq \mathcal C_1,\quad \forall\ k\in \mathbb N.
    	$$
    	Therefore,
    	$
    	||w(\cdot,t)-w^*||^2_{L^2}
    	$
    	is uniformly continuous with respect to $t$. Thus, Lemma $\ref{bdd_convergence}$ ensures that
    	$$
    	||w-w^*||_{L^2(\Omega)}\rightarrow 0\quad as\quad t\rightarrow\infty,
    	$$
    	and the desired conclusion follows from Lemma $\ref{wdecay_h}$.
    \end{proof}

    It remains to prove  Lemma \ref{uvwh}. Notice that to verify (\ref{decay1}), it suffices to show that the matrix $\mathbb P_h$ defined in (\ref{P_h-matrix}) coming from the Lyapunov functional $E_h(t)$ is positive definite.   Thus to derive the optimal results, we first leave the  coefficients $\beta_h$ and $\eta_h$  in the energy functional
        $E_h(t)$ undetermined, and then manage to explore the equivalent conditions on the parameters in the system (\ref{model})  which guarantee the  existence of positive coefficients $\beta_h$ and $\eta_h$ such that $\mathbb P_h$  is positive definite.

    \begin{proof}[Proof of Lemma $\ref{uvwh}$]
    	First of all, we claim that
    	$A_h(t),\, B_h(t)$ are nonnegative. In fact, by setting function $\mathcal I(\mathfrak u):=\mathfrak u-u^*\ln \mathfrak u$ for $\mathfrak u>0$, and using Taylor's formula, for all $x\in\Omega$ and $t>0$, there exists $\xi=\xi(x,t)\in(0,1)$ such that
    	$$
    	\begin{aligned}
    		&\mathcal I(u(x,t))-\mathcal I(u^*)\\
    		= & \mathcal I'(u^*)\cdot(u(x,t)-u^*)+\frac{1}{2}\,\mathcal I''(\,\xi u(x,t)+(1-\xi)u^*)\cdot (u(x,t)-u^*)^2\\
    		= & \dfrac{u^*}{2(\xi u(x,t)+(1-\xi)u^*)^2}(u(x,t)-u^*)^2\geq 0.
    	\end{aligned}
    	$$
    	From the computation above, we obtain that
    	$$A_h(t)=\int_\Omega \left (\mathcal I(u(x,t))-\mathcal I(u^*)\right ) \mathrm dx \geq 0.
    	$$
    	Similarly, $B_h(t)$ is also nonnegative. Since $\beta_h$ and $\eta_h$ are positive, $E_h(t)\geq 0$ for all $t\geq 0$ .
    	
    	 Next, we compute
    	$$
    	\begin{aligned}
    			\dfrac{\mathrm d}{\mathrm dt}A_h(t)=&\int_\Omega\dfrac{u-u^*}{u}u\left(1-u-a_2 v-d_1w\right)\mathrm dx\\
    	         =&\int_\Omega(u-u^*)\big  [(u^*-u)+a_2(v^*-v)+d_1(w^*-w)\big  ]\mathrm dx\\
    	         =&-\int_\Omega(u-u^*)^2\mathrm dx
    	         -a_2\int_\Omega(u-u^*)(v-v^*)\mathrm dx
    	         -d_1\int_\Omega(u-u^*)(w-w^*)\mathrm dx,
    	\end{aligned}
    	$$
    	
    	$$
    	\begin{aligned}
    			\dfrac{1}{r}\dfrac{\mathrm d}{\mathrm dt}B_h(t)=&\dfrac{1}{r}\int_\Omega\dfrac{v-v^*}{v}
    			\big[ D \nabla \cdot((1-u) \nabla v)+r v\left(1-v-a_1 u\right)-d_2w\big ]\mathrm dx\\
    	         =&-\dfrac{Dv^*}{r}\int_\Omega(1-u)\left|\dfrac{\nabla v}{v} \right|^2\mathrm dx
    	         +\int_\Omega(v-v^*)
    	         \big[ a_1(u^*-u)+(v^*-v)+\dfrac{d_2}{r}(w^*-w) \big ]\mathrm dx\\
    	         =&-a_1\int_\Omega(u-u^*)(v-v^*)\mathrm dx
    	         -\int_\Omega(v-v^*)^2\mathrm dx
    	         -\dfrac{d_2}{r}\int_\Omega(v-v^*)(w-w^*)\mathrm dx\\
    	         &-\dfrac{Dv^*}{r}\int_\Omega(1-u)\left|\dfrac{\nabla v}{v} \right|^2\mathrm dx,
    		\end{aligned}
    	$$
    	
    	$$
    	\begin{aligned}
    			\dfrac{1}{c}\dfrac{\mathrm d}{\mathrm dt}C_h(t)=&\dfrac{1}{c}\int_\Omega(w-w^*)\big[\Delta w+c(v-w) \big ]\mathrm dx\\
    			=&-\dfrac{1}{c}\int_\Omega\left| \nabla w \right|^2\mathrm dx
    			+\int_\Omega(w-w^*)\big[(v-v^*)+(w^*-w)\big ]\mathrm dx\\
    			=&-\dfrac{1}{c}\int_\Omega\left| \nabla w \right|^2\mathrm dx
    			+\int_\Omega(v-v^*)(w-w^*)\mathrm dx-\int_\Omega(w-w^*)^2\mathrm dx.
    		\end{aligned}
    	$$
    	
    	By differentiating $(\ref{eh})$ and substituting the three equations above into it, we obtain
    	\begin{eqnarray}\label{eh_differrentiate}
    			\dfrac{\mathrm d}{\mathrm dt}E_h(t) =&&\!\!\!\!\!\!\!\!
    			-\int_\Omega(u-u^*)^2\mathrm dx
    			-\beta_h\int_\Omega(v-v^*)^2\mathrm dx
    			-\eta_h\int_\Omega(w-w^*)^2\mathrm dx\cr
    			&&\!\!\!\!\!\!\!\!
    			-(a_2+a_1\beta_h)\int_\Omega(u-u^*)(v-v^*)\mathrm dx
    			-d_1\int_\Omega(u-u^*)(w-w^*)\mathrm dx\cr
    			&&\!\!\!\!\!\!\!\!
    			-(\dfrac{d_2}{r}\beta_h-\eta_h)\int_\Omega(v-v^*)(w-w^*)\mathrm dx
    			-\dfrac{D\beta_h v^*}{r}\int_\Omega(1-u)\left|\dfrac{\nabla v}{v} \right|^2\mathrm dx
    			-\dfrac{\eta_h}{c}\int_\Omega\left| \nabla w \right|^2\mathrm dx\cr
    			\leq &&\!\!\!\!\!\!\!\!
    			 -\int_\Omega\mathbf X^{\mathrm T}\mathbb P_h \mathbf X\,\mathrm dx
    			-\dfrac{\eta_h}{c}\int_\Omega\left| \nabla w \right|^2\mathrm dx,
    	\end{eqnarray}
    	where $\mathbb P_h$ and $\mathbf X$ are defined by
    	\begin{equation}\label{P_h-matrix}
    	\mathbb P_h=
    	\begin{pmatrix}\displaystyle
    	1 & \dfrac{a_2+a_1\beta_h}{2} & \dfrac{d_1}{2}\\
    	\dfrac{a_2+a_1\beta_h}{2} & \beta_h & \dfrac{\dfrac{d_2}{r}\beta_h-
    		\eta_h}{2}\\
    	\dfrac{d_1}{2} & \dfrac{\dfrac{d_2}{r}\beta_h-
    		\eta_h}{2} & \eta_h
    	\end{pmatrix},
    	\end{equation}

    	$$
    	\mathbf X=\left(u-u^*,v-v^*,w-w^*\right)^\mathbf T.
    	$$
    	
    	In order to verify $(\ref{decay1})$, we need to show that there exist positive constants $\beta_h$, $\eta_h$ such that $\mathbb P_h$ is positive definite.
       We claim that this property holds if and only if there exists a positive constant $\beta_h$ satisfying the two following inequalities simultaneously:
    	\begin{numcases}{}
    		\Phi_h(\beta_h):=
    		-a_1^2\beta_h^2+2\Big [\,2\Big (1+\dfrac{d_2}{r}\Big )-(a_1a_2+a_1d_1)\,
    		\Big ]\beta_h-(a_2+d_1)^2>0,\label{group1}\\
    		\Psi_h(\beta_h):=
    		-a_1^2\beta_h^2+2(2-a_1a_2)\beta_h-a_2^2>0.\label{group2}
    		\end{numcases}

    	Since a matrix is positive definite if and only if all its principal minors are positive, it remains to verify the positivity of every principal minor of $\mathbb P_h$. For simplicity, we denote $\alpha=\frac{1}{2}(a_2+a_1\beta_h)$.
    	First of all, we verify the first two principal minors:
    	$$
    	\mathbf {M^h_1}:=1,
    	$$
    	$$
    		\mathbf{M^h_2}:=
    	\begin{vmatrix}
    		1 & \alpha\\
    		\alpha & \beta_h
    	\end{vmatrix}
    	=\beta_h-\alpha^2
    	=\dfrac{1}{4}\left(
    	-a_1^2\beta_h^2+2(2-a_1a_2)\beta_h-a_2^2
    	\right)
    	=\dfrac{1}{4}\Psi_h(\beta_h).
    	$$
    	Thus, $(\ref{group2})$ is equivalent to	the positivity of $\mathbf{M^h_2}$.
    	Next, we consider $\det{\mathbb P_h}$:
    	
    	\begin{align}\label{quadratic_h1}
    		\det{\mathbb P_h}=&
    		\begin{vmatrix}\displaystyle
    			1 & \alpha & \dfrac{d_1}{2}\\
    			\alpha & \beta_h & \dfrac{\dfrac{d_2}{r}\beta_h-
    			\eta_h}{2}\nonumber\\
    			\dfrac{d_1}{2} & \dfrac{\dfrac{d_2}{r}\beta_h-\eta_h}{2} & \eta_h
    		\end{vmatrix}\nonumber\\
    		=&\begin{vmatrix}
    			\beta_h & \dfrac{\dfrac{d_2}{r}\beta_h-\eta_h}{2}\\
    			\dfrac{\dfrac{d_2}{r}\beta_h-\eta_h}{2} & \eta_h
    		\end{vmatrix}
    		-\alpha\begin{vmatrix}
    			\alpha & \dfrac{\dfrac{d_2}{r}\beta_h-\eta_h}{2}\\
    			\dfrac{d_1}{2} & \eta_h
    		\end{vmatrix}
    		+\dfrac{d_1}{2}
    		\begin{vmatrix}
    			\alpha & \beta_h\\
    			\dfrac{d_1}{2} & \dfrac{\dfrac{d_2}{r}\beta_h-\eta_h}{2}
    		\end{vmatrix}\nonumber\\
    		=&\dfrac{1}{4}\!\left[\!
    		-\eta_h^2\!+\!2\Big (2(\beta_h\!-\!\alpha^2)\!+\!\Big(\dfrac{d_2}{r}\beta_h\!-\!\alpha d_1\Big)\Big )\eta_h\!+\!
    		\Big (2\alpha d_1\dfrac{d_2}{r}\beta_h\!-\!\Big(\dfrac{d_2}{r}\beta_h\Big)^2\!\!-\!d_1^2\beta_h\!\Big )
    		\right].
    		\end{align}
    	Notice that $(\ref{group2})$ yields
    	$$
    	2\alpha d_1\dfrac{d_2}{r}\beta_h-\Big(\dfrac{d_2}{r}\beta_h\Big)^2-d_1^2\beta_h
    	<-\beta_h\Big(\dfrac{d_2}{r}\alpha-d_1\Big)^2\leq 0,
    	$$
    	by elementary properties of quadratic polynomial, there exists a positive constant $\eta_h$ such that $\det{\mathbb P_h}>0$ if and only if the following situations holds:
    	\begin{equation}\label{stronger1}
    		\begin{cases}
    			\Delta_h>0,\\
    			2(\beta_h-\alpha^2)+\Big(\dfrac{d_2}{r}\beta_h-\alpha d_1\Big)\geq 0,
    		\end{cases}
    	\end{equation}
    	where $\Delta_h$ is the discriminant of the quadratic $(\ref{quadratic_h1})$.
    	By calculating this discriminant and substituting $\alpha=\frac{1}{2}(a_2+a_1\beta_h)$ into it, we find
    	\begin{equation*}
    	\begin{aligned}
    		\Delta_h:=&4\left\{
    		\left[
    		(2(\beta_h-\alpha^2)+\Big (\dfrac{d_2}{r}\beta_h-\alpha d_1\Big )
    		\right]^2+
    		\Big (2\alpha d_1\dfrac{d_2}{r}\beta_h-\Big (\dfrac{d_2}{r}\beta_h\Big )^2-d_1^2\beta_h\Big )
    		\right\}\\
    		=&16\big[\,\Big (1+\dfrac{d_2}{r}\Big )\beta_h-\alpha^2-\alpha d_1-\dfrac{1}{4}d_1^2\,\big](\beta_h-\alpha^2)\\
    		=&\Bigg\{
    		-a_1^2\beta_h^2+
    		2\Big[
    		2\Big (1+\dfrac{d_2}{r}\Big )-(a_1a_2+a_1d_1)
    		\Big]\beta_h-(a_2+d_1)^2
    		\Bigg\}
    		 \\
    	&\quad\times
    	\Bigg\{
    	-a_1^2\beta_h^2+2(2-a_1a_2)\beta_h-a_2^2
    	\Bigg\}\\
    	=&\Phi_h(\beta_h)\Psi_h(\beta_h).
    		\end{aligned}
    	\end{equation*}
    	Since we already have $(\ref{group2})$, the equation above implies that $\Delta_h>0$ is equivalent to $(\ref{group1})$.
    	Also, when $\Delta_h>0$ and $\Psi_h(\beta_h)>0$, we have
    	$$
    	2(\beta_h-\alpha^2)+\Big(\dfrac{d_2}{r}\beta_h-\alpha d_1\Big)>
    	\Big (1+\dfrac{d_2}{r}\Big )\beta_h-\alpha^2-\alpha d_1>\frac{1}{4}\Phi_h(\beta_h)>0,
    	$$
    	i.e. the second inequality in $(\ref{stronger1})$ is automatically satisfied.
    	Hence, on the basis of $(\ref{group2})$, there exist positive constants $\beta_h$ and $\eta_h$ such that $\det{\mathbb P_h}>0$ if and only if $(\ref{group1})$ holds. Summing up the discussion above, our assertion has been proved.
    	
    	Now, it remains to show that under the assumptions of Theorem $\ref{thm1}$, there exists positive $\beta_h$ which satisfies $(\ref{group1})$ and $(\ref{group2})$ simultaneously. For this purpose, we denote the positive solution of $(\ref{group1})$ as $S^h_1:=\left((L^h_1)^2,\, (R^h_1)^2\right)$ and positive solution of $(\ref{group2})$ as $S^h_2:=\left( (L^h_2)^2,\,(R^h_2)^2\right)$.
    	We assume for now that we have
    	\begin{equation}\label{d2-new}
    		\dfrac{d_2}{r}<a_2+d_1-1,
    	\end{equation}
    	which is already contained in the case $(\ref{d2h1})$ and is indeed necessary since it comes from the existence and linear stability of heterogeneous steady state.
    	 Thanks to $(\ref{d2-new})$ and $a_1<1$, $S^h_1$ is not empty. On the other hand, $S^h_2$ is not empty due to $a_1a_2<1$.
    	Hence, the argument above allows us to calculate $L^h_1$, $R^h_1$, $L^h_2$ and $R^h_2$.
    	 Since $\Phi_h(\beta_h)=0$ if and only if
    	$$
    	\begin{aligned}
    		\beta_h&=\frac{1}{2a_1^2}
    	\left \{
    	2\Big[\,2\Big (1+\dfrac{d_2}{r}\Big )-a_1(a_2+d_1)\,
    		\Big ]\pm 2\sqrt{\Big[\,2\Big (1+\dfrac{d_2}{r}\Big )-a_1(a_2+d_1)\,
    		\Big ]^2-a_1^2(a_2+d_1)^2}\,
    	\right \}\\
    	&=\frac{1}{a_1^2}
    	\left \{
    	\Big (1+\dfrac{d_2}{r}\Big )+\Big (1+\dfrac{d_2}{r}-a_1(a_2+d_1)\Big )
    	\pm 2\sqrt{\,\Big (1+\dfrac{d_2}{r}\Big )\Big (1+\dfrac{d_2}{r}-a_1(a_2+d_1)\Big )}\,
    	\right \}\\
    	&=\frac{1}{a_1^2}\left \{
    	\sqrt{1+\dfrac{d_2}{r}} \pm
    	\sqrt{1+\dfrac{d_2}{r}-a_1(a_2+d_1)}
    	\right \}^2
    	\end{aligned}
    	$$
    	and $\Psi_h(\beta_h)=0$ if and only if
    	$$
    		\beta_h=\frac{1}{a_1^2}
    	\left \{
    	1+(1-a_1a_2)\pm 2\sqrt{1-a_1a_2}\,
    	\right \}=\frac{1}{a_1^2}\Big ( 1\pm\sqrt{1-a_1a_2}
    	\,\Big )^2,
    	$$
    	it follows that
    	$$
    	\begin{aligned}
    	&L^h_1=\dfrac{\sqrt{1+\dfrac{d_2}{r}}-\sqrt{1+\dfrac{d_2}{r}-a_1(a_2+d_1)}}{a_1},\
    	R^h_1=\dfrac{\sqrt{1+\dfrac{d_2}{r}}+\sqrt{1+\dfrac{d_2}{r}-a_1(a_2+d_1)}}{a_1},\\
    	&L^h_2=\dfrac{1-\sqrt{1-a_1a_2}}{a_1},\qquad\qquad\qquad\qquad\quad
    	R^h_2=\dfrac{1+\sqrt{1-a_1a_2}}{a_1}.
    	\end{aligned}
    	$$
    	Since $L^h_2<\frac{1}{a_1}<R^h_1$, to prove there is overlap part between $S^h_1$ and $S^h_2$ , we need $L^h_1 < R^h_2$, namely
    	\begin{equation}\label{ration}
    		\sqrt{1+\dfrac{d_2}{r}}-\sqrt{1+\dfrac{d_2}{r}-a_1(a_2+d_1)}
    		< 1+\sqrt{1-a_1a_2}.
    	\end{equation}
    	Recall that we also need $(\ref{d2-new})$, hence, in the following part, we verify that $(\ref{d2-new})$ and $(\ref{ration})$ hold under the assumption of Theorem $\ref{thm1}$.
    	
    	First, we consider the case $(\ref{d2h1})$. Since in this case we already have  $(\ref{d2-new})$, it remains to show that when $d_1\leq d_1^h$, where $d^h_1$ is defined in $(\ref{d1h})$, we can derive $(\ref{ration})$ from $(\ref{d2-new})$. By numerator rationalization of $(\ref{ration})$, we obtain
    	\begin{equation}\label{rationed}
    		a_1(a_2+d_1)
    		< \left(1+\sqrt{1-a_1a_2}\,\right)
    		\left\{\sqrt{1+\dfrac{d_2}{r}}+\sqrt{1+\dfrac{d_2}{r}-a_1(a_2+d_1)}\,\right\}.
    	\end{equation}
    	Substituting $\frac{d_2}{r}=a_2+d_1-1$ into the inequality above yields
    	$$
    	a_1\sqrt{a_2+d_1}
    	< \big (1+\sqrt{1-a_1a_2}\,\big )\big (1+\sqrt{1-a_1}\,\big ),
    	$$
    	which is equivalent to $d_1 < d^h_1$. Since $\frac{d_2}{r}$ is strictly smaller than $a_2+d_1-1$,  $(\ref{rationed})$ still holds when $d_1=d_1^h$.
    	Observe that the right hand side of $(\ref{rationed})$ increases in $d_2$, we obtain that $(\ref{ration})$ still holds when $(\ref{d2-new})$ is satisfied.
    	
    	Next, we demonstrate that in the case $(\ref{d2h2})$, we have $(\ref{d2-new})$ and $(\ref{ration})$. Straightforward computations show that $(\ref{ration})$ holds if and only if
    	$$
    	\sqrt{1+\frac{d_2}{r}-a_1(a_2+d_1)}
    	>
    	\frac{a_1(a_2+d_1)-(1+\sqrt{1-a_1a_2}\,)^2}{2(1+\sqrt{1-a_1a_2})}.
    	$$
    	Since $d_1\geq d_1^h$, the right hand side of the inequality above is positive. By squaring this inequality, we obtain that it is equivalent to
    	$$
    	\frac{d_2}{r}\geq
    	\left(
    	\frac{a_1(a_2+d_1)-(1+\sqrt{1-a_1a_2}\,)^2}{2(1+\sqrt{1-a_1a_2})}
    	\right)^2+a_1(a_2+d_1)-1=d_2^h,
    	$$
    	Hence, in this case we have $(\ref{ration})$. On the other hand, it follows from the discussion in $d_1\leq d_1^h$ part that when $d_1>d_1^h$, there is $d_2^h>a_2+d_1-1$. Hence, we also have $(\ref{d2-new})$.

    	Summarizing the discussion above, we draw out that  assumptions of Theorem $\ref{thm1}$ suffice to guarantee the existence of positive $\beta_h$ and $\eta_h$ such that $\mathbb P_h$ is positive definite. By the definition of positive definite matrix, there exists $\varepsilon_1>0$ such that
    	$$
    	\mathbf X^{\mathrm T}\mathbb P_h \mathbf X\geq
    	\varepsilon_1 |\mathbf X|^2.
    	$$
    	Substituting it into $(\ref{eh_differrentiate})$, we have
    	$$
    	\dfrac{\mathrm d}{\mathrm dt}E_h(t)
    		\leq -\varepsilon_1\int_\Omega|\mathbf X|^2\mathrm dx
    		-\dfrac{\eta_h}{c}\int_\Omega\left| \nabla w \right|^2\mathrm dx
    	\leq -\varepsilon_h F_h(t),
    	$$
    	where $\varepsilon_h=\min\{ \varepsilon_1,\frac{\eta_h}{c} \}$. 	
    \end{proof}

    \subsection{Auxiliary problem: systems of ODEs}\label{auxh}
    Since we have obtained the $L^\infty$ convergence of $w$ in  Lemma \ref{l2con}, there exists
    a smooth bounded positive function $\sigma(t)$,  which decays to $0$ as $t\rightarrow\infty$ and satisfies
    \begin{equation}\label{sigma}
    w^*-\sigma(t)\leq w(x,t) \leq w^*+\sigma(t),\quad x\in\Omega,\ t\geq 0.
    \end{equation}
    Then we introduce   the auxiliary ODE system  as follows:
    \begin{equation}\label{odeh}
    	\begin{cases}
    	\dfrac{\mathrm d}{\mathrm dt}\bar u_h=\bar u_h\big[1-\bar u_h-a_2\underline v_h-d_1(w^*-\sigma(t))\big], &\quad t>0,\\
    	\dfrac{\mathrm d}{\mathrm dt}\underline u_h=\underline u_h\big[1-\underline u_h-a_2\bar v_h-d_1(w^*+\sigma(t))\big], &\quad t>0,\\
    	\dfrac{\mathrm d}{\mathrm dt}\bar v_h=r\bar v_h\big[1-a_1\underline u_h-\bar v_h-\dfrac{d_2}{r}(w^*-\sigma(t))\big], &\quad t>0,\\
    	\dfrac{\mathrm d}{\mathrm dt}\underline v_h=r\underline v_h\big[1-a_1\bar u_h-\underline v_h-\dfrac{d_2}{r}(w^*+\sigma(t))\big], &\quad t>0,\\
    \end{cases}
    \end{equation}
    with initial data
    \begin{equation}\label{iniodeh}
    	\begin{aligned}
    		\bar u_h(0)=\bar u^h_0:=\max\{\max\limits_{\bar\Omega}u_0,u^*\},\quad&
    	\underline u_h(0)=\underline u^h_0:=\min \{\min \limits_{\bar\Omega}u_0,u^*\},\\
    	\bar v_h(0)=\bar v^h_0:=\max\{\max\limits_{\bar\Omega}v_0,v^*\},\quad&
    	\underline v_h(0)=\underline v^h_0:=\min \{\min \limits_{\bar\Omega}v_0,v^*\}.
    	\end{aligned}
    \end{equation}
    From $(\ref{iniodeh})$, we infer that  the initial data of $(\ref{odeh})$ satisfies
    \begin{equation}\label{ode_ini_h_compare}
    	0<\underline u^h_0\leq u^*\leq \bar u^h_0\leq 1,\quad
    	0<\underline v^h_0\leq v^*\leq \bar v^h_0 < +\infty.
    \end{equation}

    By Picard-Lindel$\ddot o$f theorem, extension theorem of solution as well as comparison theorem of ODEs, it is standard to obtain the global existence and uniqueness of solutions of $(\ref{odeh})$-($\ref{iniodeh}$) in the following lemma. We   omit its proof and refer to \cite[Lemma 3.1]{tao} for details.
    \begin{lemma}\label{positveode}
     There exists a unique global solution of $(\ref{odeh})$-$(\ref{iniodeh})$ satisfying
    	$$
    	\begin{aligned}
    			0<\bar u_h(t)\leq 1,\quad & 0<\underline u_h(t)\leq 1,\\
    			0<\bar v_h(t)\leq\max\{\bar v^h_0,1\} ,\quad & 0<\underline v_h(t)\leq 1.
    		\end{aligned}
    	$$
    \end{lemma}

Now, first we show that $u^*$, $v^*$ are constrained by the solution of $(\ref{odeh})$-$(\ref{iniodeh})$.
\begin{lemma}\label{clamp_station}
 The solution of $(\ref{odeh})$-$(\ref{iniodeh})$ satisfies
	$$
	\underline u_h (t)  \leq u^*\leq \bar u_h(t),\quad
		\underline v_h(t)   \leq v^*\leq \bar v_h(t)  , \quad t\geq 0.
	$$
\end{lemma}

\begin{proof}
	We introduce the notations
	$$
	f_+:=\max\{f,0\}\quad and \quad f_-:=\min\{f,0\},
	$$
	and they enjoy the properties that
	$$
	f_+\cdot f_-\equiv 0,\quad f\cdot f_+=f_+^2\quad and\quad f\cdot f_-=f_-^2 .
	$$
	With the notations above, it remains to show
	$$
	(u^*-\bar u_h)_+=(\underline u_h-u^*)_+=(v^*-\bar v_h)_+=(\underline v_h-v^*)_+=0,\quad t > 0.
	$$
	By the definition of $u^*,v^*,w^*$, we have
	$$
	\begin{aligned}
		\dfrac{\mathrm d}{\mathrm dt}(\bar u_h-u^*)=\bar u_h\big[(u^*-\bar u_h)+a_2(v^*-\underline v_h)+d_1\sigma(t)\big].
	\end{aligned}
	$$
	Multiplying the above equation with $-(\bar u_h-u^*)_+$, we obtain
	$$
	\begin{aligned}
		\dfrac{1}{2}\dfrac{\mathrm d}{\mathrm dt}\big[(u^*-\bar u_h)_+\big]^2
		=&\bar u_h\big[-(u^*-\bar u_h)^2_++a_2(u^*-\bar u_h)_+(\underline v_h-v^*)-d_1\sigma(t)(u^*-\bar u_h)_+\big]\\
	\leq &\bar u_h\big[-(u^*-\bar u_h)^2_++a_2(u^*-\bar u_h)_+(\underline v_h-v^*)_+\big],
	\end{aligned}
	$$
	thanks to the positivity of $\sigma(t)$. Since $\bar u_h\leq 1$, by Young's inequality, we obtain
	$$
	\dfrac{\mathrm d}{\mathrm dt}\big[(u^*-\bar u_h)_+\big]^2
	\leq \frac{a_2^2}{2}\, [(\underline v_h-v^*)_+]^2.
	$$
	In the same manner, we have
	$$
	\dfrac{\mathrm d}{\mathrm dt}\big[(\underline u_h-u^*)_+\big]^2\leq
		\frac{a_2^2}{2}\,[(v^*-\bar v_h)_+]^2,
	$$
	$$
	\dfrac{\mathrm d}{\mathrm dt}\big[(v^*-\bar v_h)_+\big]^2
	\leq \frac{r}{2}\max\{1,\bar v^h_0\}a_1^2\,[(\underline u_h-u^*)_+]^2
	$$
	and
	$$
	\dfrac{\mathrm d}{\mathrm dt}\big[(\underline v_h-v^*)_+\big]^2
	\leq \frac{r}{2}\max\{1,\bar v^h_0\}a_1^2\,[(u^*-\bar u_h)_+]^2.
	$$
	Summing up the above four inequalities together, we find
	$$
	\begin{aligned}
		&\dfrac{\mathrm d}{\mathrm dt}
		\Big\{\big[(u^*-\bar u_h)_+\big]^2+\big[(\underline u_h-u^*)_+\big]^2
		+\big[(v^*-\bar v_h)_+\big]^2+\big[(\underline v_h-v^*)_+\big]^2\Big\}\\
		\leq & \,k_0\Big\{\big[(u^*-\bar u_h)_+\big]^2+\big[(\underline u_h-u^*)_+\big]^2
		+\big[(v^*-\bar v_h)_+\big]^2+\big[(\underline v_h-v^*)_+\big]^2\Big\},
	\end{aligned}
	$$
	where $k_0=\frac{1}{2}\max\{a_2^2,\,r\max\{1,\bar v^h_0\}\,a_1^2\}$. Thanks to $(\ref{ode_ini_h_compare})$, we have
	$$
	[(u^*-\bar u^h_0)_+\big]^2=\big[(\underline u^h_0-u^*)_+\big]^2=\big[(v^*-\bar v^h_0)_+\big]^2=\big[(\underline v^h_0-v^*)_+\big]^2=0.
	$$
	By Grownwall's inequality, we obtain
	$$
	[(u^*-\bar u_h)_+\big]^2=\big[(\underline u_h-u^*)_+\big]^2=\big[(v^*-\bar v_h)_+\big]^2=\big[(\underline v_h-v^*)_+\big]^2=0,
	$$
	which ends the proof.
\end{proof}

Secondly, we  show that $(\bar u_h,\bar v_h)$ is actually the upper solution and  $(\underline u_h,\underline v_h)$ is the lower solution of $(u,v)$ in  $(\ref{model})$-$(\ref{ini})$.
\begin{lemma}\label{ode_clamp_solution}
 Suppose that the assumptions of Theorem $\ref{thm1}$ hold, $(u,v,w)$ is the global solution of $(\ref{model})$-$(\ref{ini})$, and $(\bar u_h,\underline u_h,\bar v_h,\underline v_h)$ is the solution of $(\ref{odeh})$-$(\ref{iniodeh})$, then
	$$
	\begin{aligned}
			\underline u_h(t)\leq u(x,t)\leq  \bar u_h(t),\quad x\in\Omega,\ t\geq 0,\\
		\underline v_h(t)\leq v(x,t)\leq  \bar v_h(t),\quad x\in\Omega,\ t\geq 0.
		\end{aligned}
	$$
\end{lemma}

\begin{proof}
	To simplify the notation, we introduce new variables
	$$
	\begin{aligned}
			\bar U(x,t):=\bar u_h(t)-u(x,t),\quad \underline U(x,t):=u(x,t)-\underline u_h(t),\\
			\bar V(x,t):=\bar v_h(t)-v(x,t),\quad \underline V(x,t):=v(x,t)-\underline v_h(t).
		\end{aligned}
	$$
	By the notation above, we only need to show
	$$
	\bar U_-=\underline U_-=\bar V_-=\underline V_-\equiv0,
	\quad x\in\Omega,\ t\geq 0.
	$$
	Thanks to the key property $(\ref{sigma})$, direct computations show that
	\begin{equation*}
		\begin{aligned}
			\bar U_t=&\bar U\big[
			1-\bar u_h-a_2\underline v_h-d_1(w^*-\sigma)
			\big]+u\big[\,-\bar U+a_2\underline V+d_1(w-(w^*-\sigma))\big]\\
			\geq &\big[
			1-u-\bar u_h-a_2\underline v_h-d_1(w^*-\sigma)
			\big]\bar U+a_2u\underline V.
		\end{aligned}
	\end{equation*}
	Multiplying $(\ref{barU})$ with $\bar U_-$ and using Young's inequality yields
	\begin{equation*}
		\begin{aligned}
			\dfrac{1}{2}\dfrac{\mathrm d}{\mathrm dt}(\bar U_-)^2
			\leq& \big[1-u-\bar u_h-a_2\underline v_h-d_1(w^*-\sigma)
			\big]\bar U^2_-+a_2u\bar U_-\underline V\\
			\leq& (1+d_1\sigma)\bar U^2_-+a_2u\bar U_-\underline V_-
			\leq (1+d_1\sigma+\dfrac{1}{2}u)\bar U^2_-+\dfrac{a_2^2}{2} \underline V^2_-.
		\end{aligned}
	\end{equation*}
	Integrating over $\Omega$, we have
	\begin{equation}\label{barU}
		\dfrac{1}{2}\dfrac{\mathrm d}{\mathrm dt}\int_\Omega(\bar U_-)^2\mathrm dx
		\leq \int_\Omega(2+d_1\sigma+\dfrac{1}{2}u)\bar U^2_-\,\mathrm dx
		+\dfrac{a_2^2}{2}\int_\Omega \underline V^2_-\,\mathrm dx.
	\end{equation}
	In the same manner, we obtain
	\begin{equation}
		\dfrac{1}{2}\dfrac{\mathrm d}{\mathrm dt}\int_\Omega(\underline U_-)^2\mathrm dx
		\leq \int_\Omega(2+d_1\sigma+\dfrac{1}{2}u)\bar U^2_-\,\mathrm dx
		+\dfrac{a_2^2}{2}\int_\Omega \bar V^2_-\,\mathrm dx.
	\end{equation}

	Now, we consider $\bar V_-$ and $\underline V_- $. Similarly,
	$$
	\begin{aligned}
			\bar V_t=&D\nabla \cdot((1-u) \nabla \bar V)+r\bar V\big[
			1-a_1\underline u_h-a_2\bar v_h-\dfrac{d_2}{r}(w^*-\sigma)
			\big]+rv\big(a_1\underline U-\bar V+\dfrac{d_2}{r}(w-(w^*-\sigma))\big)\\
			\geq &D\nabla \cdot((1-u) \nabla \bar V)+r\big[
			1-v-a_1\underline u_h-a_2\bar v_h-\dfrac{d_2}{r}(w^*-\sigma)
			\big]\bar V+ra_1v\underline U,
		\end{aligned}
	$$
	Multiply the inequality above with $\bar V_-$ and it follows that
	\begin{equation*}
		\begin{aligned}
			\dfrac{1}{2}\dfrac{\mathrm d}{\mathrm dt}(\bar V_-)^2
			\leq& \big[ D\nabla \cdot((1-u) \nabla \bar V)\big]\bar V_-
			+\big[1-v-a_1\underline u_h-a_2\bar v_h-\dfrac{d_2}{r}(w^*-\sigma)
			\big]\bar V^2_- + a_1v\underline U\bar V_-\\
			\leq& \big[ D\nabla \cdot((1-u) \nabla \bar V)\big]\bar V_-
			+\big[1-v+\dfrac{d_2}{r}\sigma
			\big]\bar V^2_- + a_1v\underline U_-\bar V_-\\
			\leq& \big[ D\nabla \cdot((1-u) \nabla \bar V)\big]\bar V_-
			+(1+\dfrac{d_2}{r}\sigma+\dfrac{1}{2}v^2)\bar U^2_-+\dfrac{a_1^2}{2} \underline V^2_-.
		\end{aligned}
	\end{equation*}
	By integrating over $\Omega$ and after integrating by part, we have
	\begin{eqnarray}
			\dfrac{1}{2}\dfrac{\mathrm d}{\mathrm dt}\int_\Omega(\bar V_-)^2\mathrm dx
		&\leq &- D\int_\Omega (1-u) |\nabla \bar V_-|^2\mathrm dx
		+\int_\Omega(1+\dfrac{d_2}{r}\sigma+\dfrac{1}{2}v^2)\bar V^2_-\,\mathrm dx
		+\dfrac{a_1^2}{2}\int_\Omega \underline V^2_-\,\mathrm dx\cr
		&& \leq \int_\Omega(1+\dfrac{d_2}{r}\sigma+\dfrac{1}{2}v^2)\bar V^2_-\,\mathrm dx
		+\dfrac{a_1^2}{2}\int_\Omega \underline V^2_-\,\mathrm dx.
	\end{eqnarray}
	In the same manner, we obtain
	\begin{equation}\label{underV}
	    \dfrac{1}{2}\dfrac{\mathrm d}{\mathrm dt}\int_\Omega(\underline V_-)^2\mathrm dx
		\leq \int_\Omega(1+\dfrac{d_2}{r}\sigma+\dfrac{1}{2}v^2)\underline V^2_-\,\mathrm dx
		+\dfrac{a_1^2}{2}\int_\Omega \bar V^2_-\,\mathrm dx.
	\end{equation}
	Adding $(\ref{barU})$-$(\ref{underV})$ together, due to $(\ref{v_upperbound})$ and $\sigma$ is bounded, there exists a constant $k_1>0$ such that
	\begin{equation*}
		\begin{aligned}
			\dfrac{\mathrm d}{\mathrm dt}\int_\Omega\big[
			 (\bar U_-)^2+(\underline U_-)^2+(\bar V_-)^2+(\underline V_-)^2\big]\mathrm dx
			 \leq k_1 \int_\Omega\big[
			 (\bar U_-)^2+(\underline U_-)^2+(\bar V_-)^2+(\underline V_-)^2\big]\mathrm dx.
		\end{aligned}
	\end{equation*}
	Since $(\ref{iniodeh})$ implies that
	$$
	 (\bar U_-)^2(0)=(\underline U_-)^2(0)=(\bar V_-)^2(0)=(\underline V_-)^2(0)=0,
	$$
	our conclusion comes directly after using Grownwall's inequality.
\end{proof}

\subsection{$L^\infty$ convergence of $u,v$ in the heterogeneous state}\label{ifih}
In Lemmas \ref{clamp_station} and \ref{ode_clamp_solution}, we have derived that
$$
\begin{aligned}
\underline u_h(t)\leq u(x,t),\, u^*\leq  \bar u_h(t),\quad x\in\Omega,\ t\geq 0,\\
\underline v_h(t)\leq v(x,t),\, v^*\leq  \bar v_h(t),\quad x\in\Omega,\ t\geq 0.
\end{aligned}
$$
Now we are ready to prove the $L^\infty$ convergence of $u$, $v$ to $u^*$, $v^*$ respectively.
\begin{lemma}\label{ode_decay_h}
	Suppose that the assumptions of Theorem $\ref{thm1}$ hold, $(u,v,w)$ is the solution of $(\ref{model})$-$(\ref{ini})$, then $u$ and $v$ satisfy
	$$
	||u-u^*||_{L^\infty(\Omega)}+||v-v^*||_{L^\infty(\Omega)}\rightarrow 0
		\quad as\quad t\rightarrow\infty.
	$$
\end{lemma}

\begin{proof}
	 For convenience, we turn $(\ref{odeh})$ in a form convenient to treat. Thanks to the positivity obtained in Lemma $\ref{positveode}$, we rewrite $(\ref{odeh})$ in the following form
	$$
	\begin{cases}
		\dfrac{(\bar u_h)_t}{\bar u_h}=\big[1-\bar u_h-a_2\underline v_h-d_1(w^*-\sigma(t))\big],&\quad t>0,\\
    	\dfrac{(\underline u_h)_t}{\underline u_h}=\big[1-\underline u_h-a_2\bar v_h-d_1(w^*+\sigma(t))\big],&\quad t>0,\\
    	\dfrac{(\bar v_h)_t}{\bar v_h}=r\big[1-a_1\underline u_h-\bar v_h-\dfrac{d_2}{r}(w^*-\sigma(t))\big],&\quad t>0,\\
    	\dfrac{(\underline v_h)_t}{\underline v_h}=r\big[1-a_1\bar u_h-\underline v_h-\dfrac{d_2}{r}(w^*+\sigma(t))\big],&\quad t>0,
	\end{cases}
	$$
	Straightforward computations show
	\begin{numcases}{}
			\dfrac{\mathrm d}{\mathrm dt}\ln\dfrac{\bar u_h}{\underline u_h}=-(\bar u_h-\underline u_h)+a_2(\bar v_h-\underline v_h)+2d_1\sigma(t),\quad &t>0,\label{ode1}\\
			\dfrac{\mathrm d}{\mathrm dt}\ln\dfrac{\bar v_h}{\underline v_h}=r\big[a_1(\bar u_h-\underline u_h)-(\bar v_h-\underline v_h)\big]+2d_2\sigma(t),\quad &t>0.\label{ode2}
		\end{numcases}
	Introducing the notations
	$$
	\mathcal A_0:=\dfrac{1+a_2}{(1+a_1)r},\quad
	\mathcal A_1:=\dfrac{1-a_1a_2}{1+a_2},\quad
	\mathcal A_2:=2d_1+{2d_2}\mathcal A_0,
	$$
	and adding $(\ref{ode1})$ with $\mathcal A_0\times(\ref{ode2})$, we obtain
	\begin{equation}\label{fracode}
		\dfrac{\mathrm d}{\mathrm dt}\big(\ln\dfrac{\bar u_h}{\underline u_h}+\mathcal A_0\ln\dfrac{\bar v_h}{\underline v_h}
	\big)=-\mathcal A_1\big((\bar u_h-\underline u_h)+(\bar v_h-\underline v_h)\big)+\mathcal A_2\sigma(t).
	\end{equation}
	
	First of all, we prove that there exists a constant $\kappa>0$ such that $\underline u_h\geq\kappa$ and $\underline v_h\geq\kappa$ for all $t>0$.
	By the definition of $\sigma$, there exists $T_1>0$ such that
	$$
	\sigma(t)\leq \dfrac{\mathcal A_1}{8\mathcal A_2\mathcal M}\min\{u^*,v^*\},\quad\forall\, t\geq T_1,
	$$
	where
	$$
	\mathcal M:=\max\left\{\,1,\,\max\{\mathcal A_0,\dfrac{1}{\mathcal A_0}\}\max\{\bar v^h_0,1\}\max\{\dfrac{1}{u^*},\dfrac{1}{v^*}\}
	\right\}.
	$$
	Now, we claim that there exists $T_2\geq T_1$ such that
	\begin{equation}\label{uv_ul_small}
		\big[(\bar u_h-\underline u_h)+(\bar v_h-\underline v_h)\big](T_2)
	<\dfrac{1}{4\mathcal M}\min\{u^*,v^*\}.
	\end{equation}
	Suppose that the claim is not true, i.e. for all $t\geq T_1$,
	$$
	\big[(\bar u_h-\underline u_h)+(\bar v_h-\underline v_h)\big](t)
	\geq \dfrac{1}{4\mathcal M}\min\{u^*,v^*\}.
	$$
	Then we have
	$$
	\begin{aligned}
		&-\mathcal A_1\big((\bar u_h-\underline u_h)+(\bar v_h-\underline v_h)\big)+\mathcal A_2\sigma(t)\\
	\leq& -\dfrac{\mathcal A_1}{4\mathcal M}\min\{u^*,v^*\}+
	\dfrac{\mathcal A_1}{8\mathcal M}\min\{u^*,v^*\}\\
	=& -\dfrac{\mathcal A_1}{8\mathcal M}\min\{u^*,v^*\},
	\end{aligned}
	$$
	 By substituting the above inequality into $(\ref{fracode})$, we obtain that
	$\big(\ln\frac{\bar u_h}{\underline u_h}+\mathcal A_0\ln\frac{\bar v_h}{\underline v_h}\big)$ has constant decay rate for all $t\geq T_1$, which is contradict to the fact that $\big(\ln\frac{\bar u_h}{\underline u_h}+\mathcal A_0\ln\frac{\bar v_h}{\underline v_h}\big)$ is nonnegative due to lemma $\ref{clamp_station}$, and our assertion has been proved.
	
	Once the upper and lower solutions are close enough to each other at $T_2$, namely, $(\ref{uv_ul_small})$ holds, we claim that they will not be far from each other again, more specifically, for all $t\geq T_2$,
	\begin{equation}\label{claim_uv}
		\big[(\bar u_h-\underline u_h)+(\bar v_h-\underline v_h)\big](t)
	\leq \dfrac{1}{3}\min\{u^*,v^*\}.
	\end{equation}
	In fact, if they turn to be relatively far, namely there exists $T_3>T_2$ such that
	\begin{equation}\label{equal}
		\big[(\bar u_h-\underline u_h)+(\bar v_h-\underline v_h)\big](T_3)
	=\dfrac{1}{4\mathcal M}\min\{u^*,v^*\}
	\end{equation}
	and
	$$
	\big[(\bar u_h-\underline u_h)+(\bar v_h-\underline v_h)\big](t)
	<\dfrac{1}{4\mathcal M}\min\{u^*,v^*\},\quad T_2\leq t< T_3.
	$$
	If $T_3=\infty$, then $(\ref{claim_uv})$ naturally holds.
	If $T_3<\infty$, let $T_4>T_3$ denote the maximum time such that for all $T_3<t\leq T_4$, there holds
	$$
	\big[(\bar u_h-\underline u_h)+(\bar v_h-\underline v_h)\big](t)\geq\dfrac{1}{4\mathcal M}\min\{u^*,v^*\}.
	$$
	Then, it follows directly that for all $T_3<t\leq T_4$,
	$$
	\dfrac{\mathrm d}{\mathrm dt}\left (\ln\dfrac{\bar u_h}{\underline u_h}+\mathcal A_0\ln\dfrac{\bar v_h}{\underline v_h}
	\right )(t)
	\leq -\dfrac{\mathcal A_1}{8\mathcal M}\min\{u^*,v^*\}.
	$$
	Hence, for all $T_3<t\leq T_4$, we have
	$$
	\begin{aligned}
		\big(\ln\dfrac{\bar u_h}{\underline u_h}+\mathcal A_0\ln \dfrac{\bar v_h}{\underline v_h}\big)(t)
	\leq& \big(\ln\dfrac{\bar u_h}{\underline u_h}+\mathcal A_0\ln\dfrac{\bar v_h}{\underline v_h}\big)(T_3)
	-\dfrac{\mathcal A_1}{8\mathcal M}\min\{u^*,v^*\}(t-T_3)\\
	\leq& \big(
	\ln\dfrac{\bar u_h}{\underline u_h}+\mathcal A_0\ln\dfrac{\bar v_h}{\underline v_h}\big)(T_3).
	\end{aligned}
	$$
	By the fact that $\frac{b-a}{b}\leq\ln\frac{b}{a}\leq\frac{b-a}{a}$ if $b>a>0$, we obtain that  for all $T_3<t\leq T_4$,
	
	\begin{equation}\label{sumuvode}
		\begin{aligned}
		&(\bar u_h-\underline u_h)+(\bar v_h-\underline v_h)(t)\\
	\leq&\max\{1,\dfrac{1}{\mathcal A_0}\}
	\big(\bar u_h\ln\dfrac{\bar u_h}{\underline u_h}+\mathcal A_0\bar v_h\ln \dfrac{\bar v_h}{\underline v_h}\big)(t)\\
	\leq&\max\{1,\dfrac{1}{\mathcal A_0}\}\max\{\bar v^h_0,1\}
	\big(\ln\dfrac{\bar u_h}{\underline u_h}+\mathcal A_0\ln \dfrac{\bar v_h}{\underline v_h}\big)(t)\\
	\leq&\max\{1,\dfrac{1}{\mathcal A_0}\}\max\{\bar v^h_0,1\}
	\big(\ln\dfrac{\bar u_h}{\underline u_h}+\mathcal A_0\ln \dfrac{\bar v_h}{\underline v_h}\big)(T_3)\\
	\leq&\max\{1,\dfrac{1}{\mathcal A_0}\}\max\{\bar v^h_0,1\}
	\max\{1,\mathcal A_0\}
	\Big [\,\dfrac{1}{\underline u_h}(\bar u_h-\underline u_h)+\dfrac{1}{\underline v_h}(\bar v_h-\underline v_h)\Big ](T_3)\\
	=&\max\{\mathcal A_0,\dfrac{1}{\mathcal A_0}\}\max\{\bar v^h_0,1\}
	\Big [\,\dfrac{1}{\underline u_h}(\bar u_h-\underline u_h)+\dfrac{1}{\underline v_h}(\bar v_h-\underline v_h)\Big ](T_3).
	\end{aligned}
	\end{equation}
	Notice from Lemma $\ref{clamp_station}$, $(\ref{equal})$ and the definition of $\mathcal M$ that
	\begin{equation*}
		\begin{aligned}
			\underline u_h(T_3)\geq u^*-\dfrac{1}{4\mathcal M}\min\{u^*,v^*\}\geq\dfrac{3}{4}u^*,\\
			\underline v_h(T_3)\geq v^*-\dfrac{1}{4\mathcal M}\min\{u^*,v^*\}\geq\dfrac{3}{4}v^*.
		\end{aligned}
	\end{equation*}
	Substituting the inequalities above into $(\ref{sumuvode})$, we have
	\begin{equation}\label{worsebound}
		\begin{aligned}
		&(\bar u_h-\underline u_h)(t)+(\bar v_h-\underline v_h)(t)\\
	    \leq & \dfrac{4}{3}\max\{\mathcal A_0,\dfrac{1}{\mathcal A_0}\}\max\{\bar v^h_0,1\}\max\{\dfrac{1}{u^*},\dfrac{1}{v^*}\}\big((\bar u_h-\underline u_h)+(\bar v_h-\underline v_h)\big)(T_3)\\
	    \leq & \dfrac{1}{3}\min\{u^*,v^*\},\quad T_3 < t \leq T_4.
	    \end{aligned}
	\end{equation}
	From the discussion above, we know that when $(\ref{equal})$ happens, either $(\ref{worsebound})$ holds or $(\bar u_h-\underline u_h)(t)+(\bar v_h-\underline v_h)(t)$ enjoys sharper bound $\frac{1}{4\mathcal M}\min\{u^*,v^*\}$.
	Hence, we conclude that $(\ref{claim_uv})$ holds for all $t\geq T_2$,
	
	 Thanks to $(\ref{claim_uv})$, $\underline u_h$ and $\underline v_h$ have positive lower bound $\frac{2}{3}\min\{u^*,v^*\}$ when $t\geq T_2$. Therefore, there exists a positive constant $\kappa\leq\frac{2}{3}\min\{u^*,v^*\}$ which is the lower bound of $\underline u_h$ and $\underline v_h$ uniformly in $t$, then Lemma $\ref{ode_clamp_solution}$ immediately implies that $\kappa$ is also the lower bound of $u$ and $v$.
	
	Now, we turn back to $(\ref{fracode})$ and prove the following convergence property:
	\begin{equation}\label{odeh_uv_ifi}
		||\bar u_h-\underline u_h||_{L^\infty(\Omega)}
		+||\bar v_h-\underline v_h||_{L^\infty(\Omega)}
		\rightarrow 0
		\quad as\quad t\rightarrow\infty.
	\end{equation}
	Direct computation yields that
	\begin{equation}\label{odedecay}
	\begin{aligned}
		\dfrac{\mathrm d}{\mathrm dt}\big(\ln\dfrac{\bar u_h}{\underline u_h}+\mathcal A_0\ln\dfrac{\bar v_h}{\underline v_h}\big)
	=&-\mathcal A_1\big(\underline u_h\ln\dfrac{\bar u_h}{\underline u_h}+\underline v_h\ln\dfrac{\bar v_h}{\underline v_h}\big)+\mathcal A_2\sigma(t)\\
	\leq &-2\mathcal A_3\big(\ln\dfrac{\bar u_h}{\underline u_h}+\mathcal A_0\ln\dfrac{\bar v_h}{\underline v_h}\big)+\mathcal A_2\sigma(t),
	\end{aligned}
	\end{equation}
	where $ \mathcal A_3:=\frac{1}{2}\kappa \mathcal A_1\min \{1,\frac{1}{\mathcal A_0}\}$. By comparison principle of ODE, we obtian
	$$
	\begin{aligned}
		&\big(\ln\dfrac{\bar u_h}{\underline u_h}+\mathcal A_0\ln\dfrac{\bar v_h}{\underline v_h}\big)(t)\\
	\leq & e^{-2\mathcal A_3 t}\Big\{\int_{0}^t
	\mathcal A_2\sigma(s)e^{2\mathcal A_3 s}\mathrm ds+\big(\ln\dfrac{\bar u_h}{\underline u_h}+\mathcal A_0\ln\dfrac{\bar v_h}{\underline v_h}\big)(0)
	\Big\}.
	\end{aligned}
	$$
	Thanks to L' H\^{o}pital's rule,
	$$
	\lim\limits_{t\rightarrow+\infty}
	\big(\ln\dfrac{\bar u_h}{\underline u_h}+\mathcal A_0\ln\dfrac{\bar v_h}{\underline v_h}\big)(t)
	\leq \lim\limits_{t\rightarrow+\infty}\dfrac{\mathcal A_2\sigma(t)e^{2\mathcal A_3 t}}{e^{2\mathcal A_3 t}}=0,
	$$
	and we can achieve $(\ref{odeh_uv_ifi})$ after using the inequality
	$$
	\ln\dfrac{\bar u_h}{\underline u_h}\geq\dfrac{\bar u_h-\underline u_h}{\bar u_h},\quad
	\ln\dfrac{\bar v_h}{\underline v_h}\geq\dfrac{\bar v_h-\underline v_h}{\bar v_h}.
	$$
	Combining  $(\ref{odeh_uv_ifi})$ with Lemma $\ref{clamp_station}$ and Lemma $\ref{ode_clamp_solution}$, our proof is completed.
\end{proof}

At the end, we accomplish the proof of Theorem $\ref{thm1}$.
\begin{proof}[Proof of Theorem $\ref{thm1}$]
 	First, we prove that there exists a constant $\mathcal C_0>0$ such that $E_h(t)$ and $F_h(t)$, which are defined in $(\ref{eh})$ and $(\ref{fh})$, satisfy $E_h(t)\leq \mathcal C_0 F_h(t)$. To this end, we still define function $\mathcal I(\mathfrak u):=\mathfrak u-u^*\ln \mathfrak u$ for $\mathfrak u>0$. According to L' H\^{o}pital's rule, we have
 	$$
 	\lim\limits_{\mathfrak u\rightarrow u^*}\dfrac{\mathcal I(\mathfrak u)-\mathcal I(u^*)}{\mathfrak (u-u^*)^2}=
 	\lim\limits_{\mathfrak u\rightarrow u^*}\dfrac{\mathcal I'(\mathfrak u)}{2(\mathfrak u-u^*)}
 	=\dfrac{1}{2u^*}.
 	$$
 	Thanks to Lemma $\ref{ode_decay_h}$, there exists $\mathcal T_1$ large enough such that for all $t\geq \mathcal T_1$, we have
 	\begin{equation}\label{ah_upperbound}
 		\begin{aligned}
 			&A_h(t)
 			=\int_\Omega u(x,t)-u^*-u^*\ln\dfrac{u(x,t)}{u^*}\, \mathrm dx\\
 			= & \int_\Omega \mathcal I(u(x,t))-\mathcal I(u^*)\,\mathrm dx
 			\leq\dfrac{1}{u^*}\int_\Omega \big(u(x,t)-u^*\big)^2\mathrm dx,
 		\end{aligned}
 	\end{equation}
 	as well as
 	\begin{equation}\label{ah_lowerbound}
 		A_h(t)
 		\geq\dfrac{1}{4u^*}\int_\Omega \big(u(x,t)-u^*\big)^2\mathrm dx.
 	\end{equation}
 	In the same way, enlarge $\mathcal T_1$ if necessary, for all $t\geq \mathcal T_1$ we have
 	\begin{equation}\label{bh_bound}
 		\dfrac{1}{4v^*}\int_\Omega (v-v^*)^2\,\mathrm dx
 		\leq B_h(t)\leq \dfrac{1}{v^*}\int_\Omega (v-v^*)^2\,\mathrm dx.
 	\end{equation}
 	Combining $(\ref{eh})$, $(\ref{ah_upperbound})$ and $(\ref{bh_bound})$ together, we obtain that there exists a constant $\mathcal C_0>0$ such that $E_h(t)\leq\mathcal C_0 F_h(t)$.
 	
 	Now, substituting the above inequality into $(\ref{decay1})$ yields
 	$$
 	\dfrac{\mathrm d }{\mathrm dt}E_h(t)
 	\leq -\varepsilon_h F_h(t)
 	\leq -\varepsilon_h\mathcal C_0 E_h(t),
 	 \quad t\geq \mathcal T_1.
 	$$
 	Hence, without loss of generality, there exists constants $\mathcal C_1>0$ and $\kappa_1$>0 such that
 	$$
 	E_h(t)\leq\mathcal C_1e^{-\kappa_1 t},
 	\quad t> 0.
 	$$
 	To obtain the exponential decay rate, we substitute $(\ref{ah_lowerbound})$ and the left inequality of $(\ref{bh_bound})$ into the inequality above and it follows that there exists a constant $\mathcal C_2>0$ such that
 	$$
 		||u-u^*||^2_{L^2(\Omega)}
 		+||v-v^*||^2_{L^2(\Omega)}
 		+||w-w^*||^2_{L^2(\Omega)}
 	\leq \mathcal C_2 e^{-\kappa_1 t},
 	\quad t> 0.
 	$$
 	Notice that for all $\xi\in L^{\infty}(\Omega)$, we have
 	$$
 	||\xi||_{L^{2n}(\Omega)}\leq ||\xi||^{n-1/n}_{L^{\infty}(\Omega)}||\xi||^{1/n}_{L^2(\Omega)}.
 	$$
 	By combining the two inequalities above together, we derive that there exists a constant $\mathcal C_3>0$ such that
 	\begin{equation}\label{uvw_2n}
 		||u-u^*||_{L^{2n}(\Omega)}+||v-v^*||_{L^{2n}(\Omega)}+||w-w^*||_{L^{2n}(\Omega)}
 		\leq\mathcal C_3e^{-(\kappa_1/(2n))t},
 		 \quad t> 0.
 	\end{equation}
 	
 	Now, we are ready to improve $(\ref{uvw_2n})$ to $L^\infty$ type convergence. Using  the variation-of-constants formula to the third equation of $(\ref{model})$, for each $t>2$, we can estimate $w-w^*$:
 	\begin{equation}\label{wi1i2}
 		\begin{aligned}
 			&||w(\cdot,t)-w^*||_{L^\infty(\Omega)}\\
 			\leq & ||e^{\Delta}(w(\cdot,t-1)-w^*)||_{L^\infty(\Omega)}
 			+\int_{t-1}^t
 			||e^{(t-s)\Delta}c(v(\cdot,s)-w(\cdot,s))||_{L^\infty(\Omega)}\, \mathrm d s \\
 			:=&\, \mathcal I_1+\mathcal I_2.
 		\end{aligned}
 	\end{equation}
 	Standard $L^p$-$L^q$ estimate of heat semigroup and $(\ref{uvw_2n})$ yield the existence of $\mathcal C_4>0$ such that
 	$$
 	\mathcal I_1
 		\leq \mathcal C_4(t-(t-1))^{-1/4}||w(\cdot,t)-w^*||_{L^{2n}(\Omega)}
 		\leq \mathcal C_3\mathcal C_4 e^{-(\kappa_1/(2n))t}
 	$$
 	as well as
 	$$
 	\begin{aligned}
 			\mathcal I_2
 			\leq & \int_{t-1}^t
 			\left \|e^{(t-s)\Delta}c\big((v(\cdot,s)-v^*)-(w(\cdot,s)-w^*)\big )\right \|_{L^\infty(\Omega)}\, \mathrm d s\\
 			\leq & c\,\mathcal C_4\int_{t-1}^t
 			(t-s)^{-1/4}
 			\left[\,\left \|v(\cdot,s)-v^*\right \|_{L^{2n}(\Omega)}
 			+\left \|w(\cdot,s)-w(\cdot,s))\right \|_{L^{2n}(\Omega)}\right]
 			 \mathrm d s\\
 			 \leq & c\,\mathcal C_3 \mathcal C_4 e^{-(\kappa_1/(2n))(t-1)}.
 		\end{aligned}
 	$$
 	Substituting the two inequalities above into $(\ref{wi1i2})$, we finally obtain that there exists a constant $\mathcal C_5>0$ such that
 	\begin{equation}\label{w_infi}
 		||w(\cdot,t)-w^*||_{L^\infty(\Omega)}
 	\leq \mathcal C_5 e^{-(\kappa_1/(2n))t}.
 	\end{equation}
 	
 	Now, we turn back to $(\ref{odedecay})$ and give the decay rate of $\|u-u^*\|_{L^\infty}$ and $\|v-v^*\|_{L^\infty}$. Thanks to $(\ref{w_infi})$, we could set $\sigma(t)=\mathcal C_5 e^{-(\kappa_1/(2n))t}$. Multiplying $(\ref{odedecay})$ with $e^{\mathcal A_4 t}$, where $\mathcal A_4=\min\{\frac{\kappa_1}{4n},\,\mathcal A_3\}$, for all $t>0$, we have
 	$$
	\dfrac{\mathrm d}{\mathrm dt}\Big[e^{\mathcal A_4 t}\big(\ln\dfrac{\bar u_h}{\underline u_h}+\mathcal A_0\ln\dfrac{\bar v_h}{\underline v_h}\big)\Big]
	\leq -\mathcal A_4\Big[e^{\mathcal A_4t}\big(\ln\dfrac{\bar u_h}{\underline u_h}+\mathcal A_0\ln\dfrac{\bar v_h}{\underline v_h}\big)\Big]
	+\mathcal C_5 \mathcal A_2 e^{-\mathcal A_4 t}.
	$$
 	Using comparison principle of ODE, we obtain
	$$
	\begin{aligned}
		&\Big[e^{\mathcal A_4t}\big(\ln\dfrac{\bar u_h}{\underline u_h}+\mathcal A_0\ln\dfrac{\bar v_h}{\underline v_h}\big)\Big](t)\\
	\leq & e^{-\mathcal A_4 t}\Big\{\int_0^t
	\mathcal A_2 \mathcal A_4 \mathcal C_5 \mathrm ds+\big(\ln\dfrac{\bar u_h}{\underline u_h}+\mathcal A_0\ln\dfrac{\bar v_h}{\underline v_h}\big)(0)
	\Big\},
	\end{aligned}
	$$
	which immediately implies that
	$$
	\big(\ln\dfrac{\bar u_h}{\underline u_h}+\mathcal A_0\ln\dfrac{\bar v_h}{\underline v_h}\big)(t)\rightarrow 0 \quad \text{exponentially  as} \quad t\rightarrow\infty.
	$$
	
	Using Lemma $\ref{clamp_station}$ and Lemma $\ref{ode_clamp_solution}$ as well as the lower bound of $\underline u_h$ and ${\underline v_h}$ obtained in the proof of Lemma $\ref{ode_clamp_solution}$, we obtain
	$$
	||u(\cdot,t)-u^*||_{L^\infty(\Omega)}
	+||v(\cdot,t)-v^*||_{L^\infty(\Omega)}
	\rightarrow 0 \quad \text{exponentially  as} \quad t\rightarrow\infty.
	$$	
	The proof is accomplished.
 \end{proof}

\section{The homogeneous tumor state}
In this section, we present  the proof of Theorem $\ref{thm2}$, which is about the global convergence of the  homogeneous tumor state
     $$(0,\tilde v, \tilde w) = \left( 0,\left(1+\dfrac{d_2}{r}\right)^{-1},\left(1+\dfrac{d_2}{r}\right)^{-1} \right).
     $$
The main approach of the proof is similar to that of Theorem \ref{thm1}. To avoid being redundant, we only present the details when the arguments are crucial and different.

\medskip

First of all, to prove the $L^\infty$ convergence of $w$ to $\tilde w$, the key step is the construction of  a  proper Lyapunov functional  in the following lemma, which is adjusted according to the homogeneous tumor state $(0,\tilde v, \tilde w) $  on the basis of the Lyapunov functional defined in Lemma \ref{uvwh}.

    \begin{lemma}\label{uvwc}
    	Suppose that assumptions of Theorem $\ref{thm2}$ holds, $(u,v,w)$ is the global solution of $(\ref{model})$-$(\ref{ini})$,
    	$$
    	A_c(t)=\int_\Omega u(x,t)\,\mathrm dx,
    	$$
    	
    	$$
    	B_c(t)=\int_\Omega v(x,t)-\tilde v-\tilde v\ln\dfrac{v(x,t)}{\tilde v}\,\mathrm dx,
    	$$
    	
    	$$
    	C_c(t)=\dfrac{1}{2}\int_\Omega(w(x,t)-\tilde  w)^2\,\mathrm dx.
    	$$
    	Then there exists $\beta_c>0$, $\eta_c>0$ and $\varepsilon_c>0 $ such that the functions $E_c(t)$ and $F_c(t)$ defined by
    	\begin{equation}\label{ec}
    		E_c(t)=A_c(t)+\dfrac{\beta_c}{r}B_c(t)+\dfrac{\eta_c}{c}C_c(t), \quad t>0,
    	\end{equation}
    	and
    	\begin{align}
    		F_c(t)=&
    		\int_\Omega u(x,t)^2\,\mathrm dx
    		+\int_\Omega(v(x,t)-\tilde v)^2\,\mathrm dx
    		+\int_\Omega(w(x,t)-\tilde w)^2\,\mathrm dx\nonumber\\
    		&+\int_\Omega \left|\nabla w(x,t)\right|^2\,\mathrm dx, \quad t>0,
    	\end{align}
    	satisfy
    	$$
    	E_c(t)\geq 0, \quad  t> 0,
    	$$
    	as well as
    	\begin{equation}\label{decay1c}
    		\dfrac{\mathrm d}{{\mathrm dt}}E_c(t)\leq-\varepsilon_c F_c(t).
    	\end{equation}
    \end{lemma}

    The idea of the proof of Lemma \ref{uvwc} is similar to that of Lemma $\ref{uvwh}$. We still provide the details of its proof since  it is Lemma  \ref{uvwc}  that requires the conditions imposed on the parameters in Theorem \ref{thm2}   and the computations vary due to the change of  the Lyapunov functional and the steady state.

    \begin{proof}[Proof of Lemma \ref{uvwc}]
    	Similar to Lemma $\ref{uvwh}$, $A_c(t)$, $B_c(t)$ and $C_c(t)$ are nonnegative. For convenience, we denote
    	$$\delta:=(a_2+d_1)\left({\dfrac{d_2}{r}+1}\right)^{-1}.$$
    	We assume for now that we have
    	\begin{equation}\label{delta>1}
    		\delta>1
    	\end{equation}
    	and
    	\begin{equation}\label{a1a2delta}
    		a_1a_2<\delta.
    	\end{equation}
    	In the later part of this proof, we will show that $(\ref{delta>1})$ and $(\ref{a1a2delta})$ are actually contained in the cases $(\ref{thm21})$-$(\ref{thm24})$.
    	In fact, $(\ref{delta>1})$, namely $\frac{d_2}{r}<a_2+d_1-1$ comes from the linear stability of homogeneous tumor state. $(\ref{delta>1})$ is already contained in the cases $(\ref{thm21})$ and $(\ref{thm22})$, and we will verify $(\ref{delta>1})$ in the cases $(\ref{thm23})$ and $(\ref{thm24})$ in the comparison between $d_2^c$ and $a_2+d_1-1$ in the following part of this proof.
    	On the other hand, as for property $(\ref{a1a2delta})$, in the case $(\ref{thm21})$, $(\ref{a1a2delta})$ is contained in its last inequality. In the case $(\ref{thm22}) $ and case $(\ref{thm23})$, thanks to $a_1a_2<1$ and $(\ref{delta>1})$, $(\ref{a1a2delta})$ automatically holds. We will verify that we have $(\ref{a1a2delta})$ in the case $(\ref{thm24})$ in later part.
    	
    	Due to the fact that $u\leq 1$, we have
    	\begin{equation}\label{dac}
    		\begin{aligned}
    			\dfrac{\mathrm d}{\mathrm dt}A_c(t)=&\int_\Omega u\left(1-u-a_2 v-d_1w\right)\mathrm dx\\
    	         =&-(\delta-1)\int_\Omega u\,\mathrm dx
    	         -\int_\Omega u^2\,\mathrm dx
    	         -a_2\int_\Omega u(v-\tilde v)\mathrm dx
    	         -d_1\int_\Omega u(w-\tilde w)\mathrm dx\\
    	         \leq& -\delta\int_\Omega u^2\,\mathrm dx
    	         -a_2\int_\Omega u(v-\tilde v)\mathrm dx
    	         -d_1\int_\Omega u(w-\tilde w)\mathrm dx,
    		\end{aligned}
    	\end{equation}
    	Similarly,
    		\begin{eqnarray}\label{dbc}
    			\dfrac{1}{r}\dfrac{\mathrm d}{\mathrm dt}B_c(t)
    			=&&\!\!\!\!\!\!\!\!
    			 \dfrac{1}{r} \int_\Omega\dfrac{v-\tilde v}{v}
    			\left[ D \nabla \cdot((1-u) \nabla v)
    			+r v\left(1-v-a_1 u\right)-d_2w\right]\mathrm dx\nonumber\\
    	         =&&\!\!\!\!\!\!\!\!
    	         -\dfrac{D\tilde v}{r}\int_\Omega(1-u)\left|\dfrac{\nabla v}{v} \right|^2\mathrm dx
    	         +\int_\Omega(v-\tilde v)
    	         \left[ -a_1u+(\tilde v-v)+\dfrac{d_2}{r}(\tilde w-w) \right]\mathrm dx\nonumber\\
    	         =&&\!\!\!\!\!\!\!\!
    	         -a_1\int_\Omega u(v-\tilde v)\mathrm dx
    	         -\int_\Omega(v-\tilde v)^2\mathrm dx
    	         -\dfrac{d_2}{r}\int_\Omega(v-\tilde v)(w-\tilde w)\mathrm dx\nonumber\\
    	         &&\!\!\!\!\!\!\!\!
    	         -\dfrac{D\tilde v}{r}\int_\Omega(1-u)\left|\dfrac{\nabla v}{v} \right|^2\mathrm dx,
    		\end{eqnarray}
    	\begin{equation}\label{dcc}
    		\begin{aligned}
    			\dfrac{1}{c}\dfrac{\mathrm d}{\mathrm dt}C_c(t)=&\dfrac{1}{c}\int_\Omega(w-\tilde w)\left[\Delta w+c(v-w) \right]\mathrm dx\\
    			=&-\dfrac{1}{c}\int_\Omega\left| \nabla w \right|^2\mathrm dx
    			+\int_\Omega(w-\tilde w)[(v-\tilde v)+(\tilde w-w)]\mathrm dx\\
    			=&-\dfrac{1}{c}\int_\Omega\left| \nabla w \right|^2\mathrm dx
    			+\int_\Omega(v-\tilde v)(w-\tilde w)-\int_\Omega(w-\tilde w)^2\mathrm dx.
    		\end{aligned}
    	\end{equation}
    	
    	By differentiating $(\ref{ec})$ and substituting $(\ref{dac})$-$(\ref{dcc})$ into it, we obtain
    		\begin{eqnarray}\label{ec_differentiate}
    			\dfrac{\mathrm d}{\mathrm dt}E_c(t)=&&\!\!\!\!\!\!\!\!
    			-\delta\int_\Omega u^2\mathrm dx
    			-\beta_c\int_\Omega(v-\tilde v)^2\mathrm dx
    			-\eta_c\int_\Omega(w-\tilde w)^2\mathrm dx\nonumber\\
    			&&\!\!\!\!\!\!\!\!
    			-(a_2+a_1\beta_c)\int_\Omega u(v-\tilde v)\mathrm dx
    			-d_1\int_\Omega u(w-\tilde w)\mathrm dx
    			-(\dfrac{d_2}{r}\beta_c-\eta_c)\int_\Omega(v-\tilde v)(w-\tilde w)\mathrm dx\nonumber\\
    			&&\!\!\!\!\!\!\!\!
    			-\dfrac{D\beta_c \tilde v}{r}\int_\Omega(1-u)\left|\dfrac{\nabla v}{v} \right|^2\mathrm dx
    			-\dfrac{\eta_c}{c}\int_\Omega\left| \nabla w \right|^2\mathrm dx\nonumber\\
    			\leq &&\!\!\!\!\!\!\!\!
    			-\int_\Omega\mathbf Y^{\mathrm T}\mathbb P_c \mathbf Y\,\mathrm dx
    			-\dfrac{\eta_c}{c}\int_\Omega\left| \nabla w \right|^2\mathrm dx,
    		\end{eqnarray}
    	where
    	$$
    	\mathbb P_c=
    		\begin{pmatrix}\displaystyle
    			\delta & \dfrac{a_2+a_1\beta_c}{2} & \dfrac{d_1}{2}\\
    			\dfrac{a_2+a_1\beta_c}{2} & \beta_c & \dfrac{\dfrac{d_2}{r}\beta_c-
    			\eta_c}{2}\\
    			\dfrac{d_1}{2} & \dfrac{\dfrac{d_2}{r}\beta_c-
    			\eta_c}{2} & \eta_c
    		\end{pmatrix}
    	$$
    	and
    	$$
    	\mathbf Y=\left(u,v-\tilde v,w-\tilde w\right)^\mathbf T.
    	$$
    	
    	To verify $(\ref{decay1c})$, we need to show that there exist positive constants $\beta_c$, $\eta_c$ such that $\mathbb P_c$ is positive definite.
       We declare that the above property holds if and only if there exists $\beta_c>0$ satisfying the two following inequalities simultaneously:
    	\begin{numcases}{}
    		\Phi_c(\beta_c):=-a_1^2\beta_c^2+
    		2(2-a_1)(a_2+d_1)\beta_c-(a_2+d_1)^2>0,\label{group1c}\\
    		\Psi_c(\beta_c):=
    		-a_1^2\beta_c^2+2(2\delta -a_1a_2)\beta_c-a_2^2>0\label{group2c},
    		\end{numcases}    	
    	For simplicity, we still denote $\alpha:=\tfrac{1}{2}(a_2+a_1\beta_c)$.
    	Similar to the approach we used in Lemma $\ref{uvwh}$, we consider every principal minor of $\mathbb P_c$:
    	$$
    	\mathbf {M^c_1}:=1,
    	$$
    	$$
    	\mathbf{M^c_2}:=
    	\begin{vmatrix}
    		\delta & \alpha\\
    		\alpha & \beta_c
    	\end{vmatrix}
    	=\delta\beta_c-\alpha^2
    	=\dfrac{1}{4}\left(
    	-a_1^2\beta_c^2+2(2\delta -a_1a_2)\beta_c-a_2^2\,
    	\right)=\dfrac{1}{4}\Psi_c(\beta_c),
    	$$
    hence, $(\ref{group2c})$ holds if and only if $\mathbf{M^c_2}>0$.

    Now, we consider the determinant of $\mathbb P_c$ and find out the relationship between $\det{\mathbb P_c}>0$ and $(\ref{group1c})$-$(\ref{group2c})$.
    	\begin{align}
    	    \!\!\!\!\!\!\!\!\det{\mathbb P_c}&=
    		\begin{vmatrix}\displaystyle
    			\delta & \alpha & \dfrac{d_1}{2}\\
    			\alpha & \beta_c & \dfrac{\dfrac{d_2}{r}\beta_c-
    			\eta_c}{2}\\
    			\dfrac{d_1}{2} & \dfrac{\dfrac{d_2}{r}\beta_c-\eta_c}{2} & \eta_c
    		\end{vmatrix}\nonumber\\
    		&=\delta\begin{vmatrix}
    			\beta_c & \dfrac{\dfrac{d_2}{r}\beta_c-\eta_c}{2}\\
    			\dfrac{\dfrac{d_2}{r}\beta_c-\eta_c}{2} & \eta_c
    		\end{vmatrix}
    		-\alpha\begin{vmatrix}
    			\alpha & \dfrac{\dfrac{d_2}{r}\beta_c-\eta_c}{2}\\
    			\dfrac{d_1}{2} & \eta_c
    		\end{vmatrix}
    		+\dfrac{d_1}{2}
    		\begin{vmatrix}
    			\alpha & \beta_c\\
    			\dfrac{d_1}{2} & \dfrac{\dfrac{d_2}{r}\beta_c-\eta_c}{2}
    		\end{vmatrix}\nonumber\\
    		&=\!\dfrac{1}{4}\!\left[
    		-\delta\eta_c^2
    		\!+\!2\Big(2(\delta\beta_c\!-\!\alpha^2)\!+\!\Big(\dfrac{d_2}{r}\delta\beta_c\!-\!\alpha d_1\Big)\!\Big)\eta_c
    		\!+\!\!\Big(2\alpha d_1\dfrac{d_2}{r}\beta_c\!-\!\delta \Big(\dfrac{d_2}{r}\beta_c\Big)^2\!\!-\!d_1^2\beta_c\!\Big)
    		\right]\!.\!\label{quadratic_c}
    		\end{align}
    	Notice from $(\ref{group2c})$ that
    	$$
    	2\alpha d_1\dfrac{d_2}{r}\beta_c-\delta \Big(\dfrac{d_2}{r}\beta_c\Big)^2-d_1^2\beta_c
    	<-\beta_c\Big(\dfrac{d_2}{r}\alpha-d_1\Big)^2 \leq 0,
    	$$
    	fundamental properties of quadratic polynomials implies that there exists $\eta_c>0$ such that $\det{\mathbb P_c}>0$ if and only if the following situation holds:
    	\begin{equation}\label{stronger_c}
    		\begin{cases}
    			{\Delta_c}>0,\\
    			2(\delta\beta_c-\alpha^2)+\Big(\dfrac{d_2}{r}\delta\beta_c-\alpha d_1\Big)\geq 0,
    		\end{cases}
    	\end{equation}
    	where ${\Delta_c}$ is the discriminant of the quadratic $(\ref{quadratic_c})$. By calculating this discriminant and substituting $\alpha=\tfrac{1}{2}(a_2+a_1\beta_c)$ into it, we find
    	\begin{equation*}
    	\begin{aligned}
    		{\Delta_c}=&4\left\{
    		\left[
    		(2(\delta\beta_c-\alpha^2)+\Big(\delta \dfrac{d_2}{r}\beta_c-\alpha d_1\Big)
    		\right]^2+
    		\delta\Big (2\alpha d_1\dfrac{d_2}{r}\beta_c-\Big(\delta \dfrac{d_2}{r}\beta_c\Big)^2-d_1^2\beta_c\Big )
    		\right\}\\
    		=&16
    		\left[\,\delta \Big (1+\dfrac{d_2}{r}\Big )\beta_c-\alpha^2-\alpha d_1-\dfrac{1}{4}d_1^2\,\right ](\delta\beta_c-\alpha^2)\\
    		=&\Big\{ -a_1^2\beta_c^2+
    		2(2-a_1)(a_2+d_1)\beta_c-(a_2+d_1)^2\Big\}\times
    		\Big\{ -a_1^2\beta_c^2+2(2\delta -a_1a_2)\beta_c-a_2^2 \Big\}\\
    		=& \Phi_c(\beta_c)\Psi_c(\beta_c).
    		\end{aligned}
    	\end{equation*}
    	Since we already have $(\ref{group2c})$, it follows from the equations above that ${\Delta_c}>0$ if and only if $(\ref{group1c})$ holds.
    	Also, when $\Delta_c>0$ and $\Psi_c(\beta_c)>0$,
    	$$
    	2(\delta\beta_c-\alpha^2)+\Big(\dfrac{d_2}{r}\delta\beta_c-\alpha d_1\Big)>\delta\Big(1+\frac{d_2}{r}\Big)\beta_c-\alpha^2 -\alpha d_1>\dfrac{1}{4}\Phi_c(\beta_c)>0,
    	$$
    	which implies the second equation in $(\ref{stronger_c})$ is automatically satisfied.
    	Therefore, on the basis of $(\ref{group2c})$, there exist $\beta_c>0$ and $\eta_c>0$ such that $\det{\mathbb P_c}>0$ if and only if $(\ref{group1c})$ holds. Summing up the discussion above, our assertion has been proved.
    	
    	 Now, it remains to show that under the assumptions of Theorem $\ref{thm2}$, there exists  $\beta_c>0$ which satisfies $(\ref{group1c})$ and $(\ref{group2c})$ simultaneously. For convenience, we denote the positive solution of $(\ref{group1c})$ as $S^c_1:=\left((L^c_1)^2,\, (R^c_1)^2\right)$ and positive solution of $(\ref{group2c})$ as $S^c_2:=\left( (L^c_2)^2,\,(R^c_2)^2\right)$.
    	 Thanks to $a_1<1$ and $(\ref{a1a2delta})$, $S^c_1$ and $S^c_2$ are not empty.
    	 Direct calculations show that
    	$$
    	\begin{aligned}
    	&L^c_1=\dfrac{1-\sqrt{1-a_1}}{a_1}\sqrt{a_2+d_1} ,\quad
    	R^c_1=\dfrac{1+\sqrt{1-a_1}}{a_1}\sqrt{a_2+d_1},\\
    	&L^c_2=\dfrac{1}{a_1}(\sqrt{\delta}-\sqrt{\delta-a_1a_2}) ,\quad \,\,
    	R^c_2=\dfrac{1}{a_1}(\sqrt{\delta}+\sqrt{\delta-a_1a_2}).
    	\end{aligned}
    	$$
    	By the fact that
    	$$
    	 R^c_1 > \frac{\sqrt{a_2+d_1}}{a_1} > \frac{\sqrt{\delta}}{a_1} > L^c_2,
    	$$
    	$S^c_1$ and $S^c_2$ have overlap part if and only if $L^c_1 < R^c_2$, namely
    	\begin{equation}\label{inequality_c}
    		(1-\sqrt{1-a_1})\sqrt{a_2+d_1}
    		< \sqrt{\delta}+\sqrt{\delta-a_1a_2}.
    	\end{equation}
    	Therefore, it remains to show that in each case of Theorem $\ref{thm2}$, we have $(\ref{delta>1})$, $(\ref{a1a2delta})$ and $(\ref{inequality_c})$.
    	
    	First, we investigate the case that $d_1\leq d_{1}^{c}$. By numerator rationalization of $(\ref{inequality_c})$, we have
    	$$
    	\sqrt{\delta}<\frac{a_1a_2}{(1-\sqrt{1-a_1})\sqrt{a_2+d_1}}+\sqrt{\delta-a_1a_2}.	
    	$$
    	Then, squaring both sides of the above inequality yields
    	\begin{equation}\label{rationed-c}
    		a_1a_2-\left(\frac{a_1a_2}{(1-\sqrt{1-a_1})\sqrt{a_2+d_1}}\right)^2
    	<\frac{2a_1a_2 \sqrt{\delta-a_1a_2}}{(1-\sqrt{1-a_1})\sqrt{a_2+d_1}}.
    	\end{equation}
    	Thanks to $d_1\leq d_1^c$, the left hand side of the above inequality is nonpositive, and the inequality automatically holds. Hence, we have verified the case $(\ref{thm21})$.
    	
    	Next, we turn to the case that $d_1 > d_{1}^{c}$.
    	Since if we have $(\ref{rationed-c})$, then $(\ref{a1a2delta})$ directly holds.
    	Direct computations show that $(\ref{inequality_c})$, namely $(\ref{rationed-c})$ is equivalent to
    	$$
    	    \dfrac{d_2}{r}
    		<4\left (1-\sqrt{1-a_1} +\dfrac{1}{1-\sqrt{1-a_1}}\dfrac{a_1a_2}{a_2+d_1}
    		\right)^{-2} -1
    		=d^c_2.
    	$$
    	Since we also need $d_2$ satisfies $(\ref{delta>1})$, it is now necessary to consider what kind of $d_1$ satisfies $d^c_2\geq a_2+d_1-1$ so that not $\frac{d_2}{r} < d^c_2$ but $\frac{d_2}{r}<a_2+d_1-1$ is the crucial restriction. It is not difficult to verify that
    	when $a_1a_2\geq 1$, we have $d^c_2 \leq  a_2+d_1-1$.  Hence, the assumptions in the case $(\ref{thm24})$ is sufficient to obtain $(\ref{delta>1})$, $(\ref{a1a2delta})$ and $(\ref{inequality_c})$.
    	When $a_1a_2<1$, $d^c_2\geq a_2+d_1-1$ is equivalent to
    	\begin{equation}\label{d1}
    		\left(\dfrac
    	{1-\sqrt{1-a_1a_2}}
    	{1-\sqrt{1-a_1}}\right)^2-a_2
    	\leq d_1\leq
    	\left(\dfrac
    	{1+\sqrt{1-a_1a_2}}
    	{1-\sqrt{1-a_1}}\right)^2-a_2=d_1^h,
    	\end{equation}
    	where $d_1^h$ is defined in $(\ref{d1h})$.
    	Since $a_1a_2<1$, we have
    	$$
    	\left(\dfrac
    	{1-\sqrt{1-a_1a_2}}
    	{1-\sqrt{1-a_1}}\right)^2-a_2<d_1^c,
    	$$
    	and the restriction of $d_1$ on the left side of $(\ref{d1})$ can be neglected.
        Therefore, when $a_1a_2<1$, if $d_1^h<d_1 \leq d_1^c$, we need $\frac{d_2}{r}<a_2+d_1-1$, if $d_1 > d_1^c$, we need $\frac{d_2}{r}<d_2^c$, and they are the cases $(\ref{thm22})$ and $(\ref{thm23})$ respectively.
    	
    	Summarizing the discussion above, we draw out that under the assumptions of Theorem $\ref{thm2}$, there exist $\beta_c>0$ and $\eta_c>0$ such that $\mathbb P_c$ is positive definite. Namely, there exists a constant $\varepsilon_2>0$ such that
    	$$
    	\mathbf Y^{\mathrm T}\mathbb P_c \mathbf Y\geq
    	\varepsilon_2 |\mathbf Y|^2.
    	$$
    	Substituting it into $(\ref{ec_differentiate})$, we have
    	\begin{equation}\label{ecinquality}
    		\dfrac{\mathrm d}{\mathrm dt}E_c(t)
    		\leq -\varepsilon_2\int_\Omega|\mathbf Y|^2 \mathrm dx
    		-\dfrac{\eta_c}{c}\int_\Omega\left| \nabla w \right|^2 \mathrm dx
    	\leq -\varepsilon_c F_c(t),
    	\end{equation}
    	where $\varepsilon_c=\min\{ \varepsilon_2,\frac{\eta_c}{c} \}$.
    \end{proof}

  On the basis of  Lemma \ref{uvwc},   the $L^\infty$ convergence of $w$ to $\tilde w$ can be verified as follows.
    \begin{lemma}\label{l2contumor}
    	Suppose that the  assumptions of Theorem $\ref{thm2}$ hold and $(u,v,w)$ is the global solution of $(\ref{model})$-$(\ref{ini})$, then
    	$$
    	||w-\tilde w||_{L^\infty(\Omega)}\rightarrow 0\quad as\quad t\rightarrow\infty.
    	$$
    \end{lemma}

The proof is omitted since it is the same as  that of Lemma \ref{l2con}.
Then thanks to   Lemma \ref{l2contumor}, there exists
a smooth bounded positive function $\gamma(t)$,  which decays to $0$ as $t\rightarrow\infty$ and satisfies
$$
\tilde w-\gamma(t)\leq w(x,t) \leq \tilde w+\gamma(t),\quad x\in\Omega,\ t \geq 0.
$$
and  the auxiliary ODE system is introduced as follows:
     \begin{equation}\label{odec}
    	\begin{cases}
    	\dfrac{\mathrm d}{\mathrm dt}\bar u_c=\bar u_c\big[1-\bar u_c-a_2\underline v_c-d_1(\tilde w-\gamma(t))\big],&\quad t>0,\\
    	\dfrac{\mathrm d}{\mathrm dt}\bar v_c=r\bar v_c\big[1-\bar v_c-\dfrac{d_2}{r}(\tilde w-\gamma(t))\big],&\quad t>0,\\
    	\dfrac{\mathrm d}{\mathrm dt}\underline v_c=r\underline v_c\big[1-a_1\bar u_c-\underline v_c-\dfrac{d_2}{r}(\tilde w+\gamma(t))\big],&\quad t>0,\\
    \end{cases}
    \end{equation}
    with initial data
    \begin{equation}\label{iniodec}
    	\begin{aligned}
    		&\bar u_c(0)=\bar u^c_0:=\max\limits_{\bar\Omega}u_0,\quad\\
    	    &\bar v_c(0)=\bar v^c_0:=\max\{\max\limits_{\bar\Omega}v_0,\tilde v\},\quad
    	     \underline v_c(0)=\underline v^c_0:=\min \{\min \limits_{\bar\Omega}v_0,\tilde v\}.
    	\end{aligned}
    \end{equation}
    From $(\ref{iniodec})$, we infer that the initial data of $(\ref{odec})$ satisfies
    \begin{equation}
    	0<  \bar u^c_0\leq 1,\,\quad
    	0<\underline v^c_0\leq \tilde v\leq \bar v^c_0 < +\infty.
    \end{equation}

 Parallel to Lemmas \ref{positveode}, \ref{clamp_station}, \ref{ode_clamp_solution} in Section 3.2  and  Lemma \ref{ode_decay_h}  in Section 3.3,  in the following lemmas, we  derive some estimates related to the auxiliary ODE system $(\ref{odec})$-$(\ref{iniodec})$ and then verify the $L^\infty$ convergence of $u$, $v$ to $0$, $\tilde v$ respectively.


    \begin{lemma}\label{ode-tumor}
    	The auxiliary ODE system $(\ref{odec})$-$(\ref{iniodec})$ admits a unique global solution carrying the property
    	\begin{equation*}
    		\begin{aligned}
    			0<\bar u_c(t)\leq 1,\quad
    			0<\bar v_c(t)\leq\max\{\bar v^c_0,1\},\quad
    			0<\underline v_c(t)\leq 1,\ \ \ t\geq 0.
    		\end{aligned}
    	\end{equation*}
    \end{lemma}

    \begin{lemma}\label{clamp_station_c}
    	The solution of $(\ref{odec})$-$(\ref{iniodec})$ satisfies
	\begin{equation*}
		\underline v_c(t)\leq \tilde v\leq \bar v_c(t), \quad \, t \geq 0.
	\end{equation*}
    \end{lemma}

    \begin{lemma}\label{ode_clamp_solution_c}
    	 Suppose that the assumptions of Theorem $\ref{thm2}$ hold. Let $(u,v,w)$ be the solution of $(\ref{model})$-$(\ref{ini})$, and $(\bar u_c,\bar v_c,\underline v_c)$ be the solution of $(\ref{odec})$-$(\ref{iniodec})$, then
	\begin{equation*}
		\begin{aligned}
		0  \leq	u(x,t)\leq  \bar u_c(t),\quad x\in\Omega,\ t \geq 0,\\
		\underline v_c(t)\leq v(x,t)\leq  \bar v_c(t),\quad x\in\Omega,\ t \geq 0 .
		\end{aligned}
	\end{equation*}
    \end{lemma}

 \begin{lemma}\label{ode_decay_c}
 	Suppose that the assumptions of Theorem $\ref{thm2}$ hold, and $(u,v,w)$ is the global solution of $(\ref{model})$-$(\ref{ini})$. Then
 	\begin{equation*}
 	||u||_{L^\infty(\Omega)}
 	+||v-\tilde v||_{L^\infty(\Omega)}\rightarrow 0
 	\quad as\quad t\rightarrow\infty.
 	\end{equation*}
 \end{lemma}

The proofs of Lemmas \ref{ode-tumor},  \ref{clamp_station_c} and  \ref{ode_clamp_solution_c} are omitted since they are similar to  those of Lemmas \ref{positveode}, \ref{clamp_station} and \ref{ode_clamp_solution}  respectively and simpler.
However, the proof of  Lemma \ref{ode_decay_h}  cannot be applied to Lemma \ref{ode_decay_c} since for the  homogeneous tumor state
$(0,\tilde v, \tilde w)$,   $(0, \underline v_c)$ is  the the lower solution of $(u,v)$ as demonstrated in Lemma   \ref{ode_clamp_solution_c}.
Our strategy here is to use the results of ODE competitive systems.

\begin{proof}[Proof of Lemma \ref{ode_decay_c}]
    Thanks to the Lemmas \ref{clamp_station_c} and  $\ref{ode_clamp_solution_c}$,  it suffices  to show
    $$
    ||\bar u_c||_{L^\infty(\Omega)}
    +||\bar v_c-\underline v_c||_{L^\infty(\Omega)}\rightarrow 0
		\quad as\quad t\rightarrow\infty,
    $$
    where $(\bar u_c,\bar v_c,\underline v_c)$ is the solution of $(\ref{odec})$-$(\ref{iniodec})$.
	Notice that the equation of $\bar v_c$ in $(\ref{odec})$ is independent to $\bar u_c$ and $\underline v_c$, we directly obtain that $\bar v_c(t)\rightarrow \tilde v$ as $t\rightarrow+\infty$. Hence, we only need to prove that
	 for any $0<\varepsilon<\varepsilon_0$, where $\varepsilon_0=\varepsilon_0(a_1,a_2,d_1,d_2,r)$ is small enough,
	there exists $\mathcal T=\mathcal T(\varepsilon)$ such that $\forall\, t>\mathcal T$,
	\begin{equation}\label{uvc_epsilon}
		||\bar u_c||_{L^\infty(\Omega)}
		+||\tilde v -\underline v_c||_{L^\infty(\Omega)}<\varepsilon.
	\end{equation}
	Thanks to the selection of $\gamma(t)$, we know that for all $\varepsilon>0$, there exists $\mathcal T'>0$ such that for all $t>\mathcal T'$, $\gamma(t)\leq\frac{1}{2}\left(\min\{\,d_1,\frac{d_2}{r}\}\right)^{-1}\varepsilon$. Treating $\mathcal T'$ as the initial time, we let $( \bar{\mathfrak u}, \underline{\mathfrak v})$ be the solution of the following ODE system:
	\begin{equation}\label{uv_f}
		\begin{cases}
			\dfrac{\mathrm d}{\mathrm dt}\bar{\mathfrak u}=\bar{\mathfrak u}\big[1-\bar{\mathfrak u}-a_2\underline{\mathfrak v}-d_1\tilde w+\dfrac{1}{2}\varepsilon\big],&\quad t>\mathcal T' ,\\
    	\dfrac{\mathrm d}{\mathrm dt}\underline{\mathfrak v}=r\underline{\mathfrak v}\big[1-a_1\bar{\mathfrak u}-\underline{\mathfrak v}-\dfrac{d_2}{r}\tilde w-\dfrac{1}{2}\varepsilon\big],&\quad t>\mathcal T',
		\end{cases}
	\end{equation}
	with initial data
	$$
	\bar{\mathfrak u}(\mathcal T')=\bar u_c(\mathcal T'),\quad \underline{\mathfrak v}(\mathcal T')=\underline v_c(\mathcal T').
	$$
	 Since $\bar u_c$ and $\underline v_c$ in $(\ref{odec})$ form a competitive ODE system, it is not difficult to verify that $\bar{\mathfrak u}(t)\geq\bar u_c(t)$ and $\underline{\mathfrak v}(t)\leq\underline v_c(t)$ for all $t\geq\mathcal T'$. Hence, to prove $(\ref{uvc_epsilon})$, it suffice to demonstrate that there exists $\mathcal T=\mathcal T(\varepsilon)>\mathcal T'$ such that $\forall\, t>\mathcal T$,
	 $$
	 ||\bar{\mathfrak u}(t)||_{L^\infty}+||\underline{\mathfrak v}(t)-\tilde v||_{L^\infty}<\varepsilon.
	 $$
	 Denote two straight lines on $(\bar{\mathfrak u},\underline{\mathfrak v})$ plane
	 $$
	 \begin{aligned}
	 	&\mu:
	    1-\bar{\mathfrak u}-a_2\underline{\mathfrak v}-d_1\tilde w+\frac{1}{2}\varepsilon=0,\\
	    &\nu:
	    1-a_1\bar{\mathfrak u}-\underline{\mathfrak v}-\frac{d_2}{r}\tilde w-\frac{1}{2}\varepsilon=0.
	 \end{aligned}
	 $$
	
	 First of all, we show that under the assumptions of Theorem $\ref{thm2}$, system $(\ref{uv_f})$ does not have positive steady state.
	 If $a_1a_2 = 1$, then $\mu$ and $\nu$ are parallel.
	 If $a_1a_2 \not = 1$, direct computations show that $\mu$ and $\nu$ intersect at $(\mathfrak u_0,\mathfrak v_0)$, where
	 $$
	 \mathfrak u_0=
	 \frac{1}{1-a_1a_2}
	 \left (
	 1-\delta+\frac{1}{2}(a_2+1)\varepsilon
	  \right ),
	 $$
	 $$
	 \mathfrak v_0=
	 \left (1+\frac{d_2}{r}\right )^{-1}
	 \left[1+\frac{a_1}{a_1a_2-1}\left(1+\frac{d_2}{r}-a_2-d_1\right )
	 +\frac{2a_1a_2+a_1-1}{2(1-a_1a_2)}\varepsilon
	 \right].
	 $$
	 We can immediately obtain that when $\varepsilon_0$ is sufficiently small and $a_1a_2<1$, there is $\mathfrak u_0<0$. Hence, in the case $(\ref{thm22})$ and $(\ref{thm23})$, there is no positive steady state. It remains to consider the cases $(\ref{thm21})$ and $(\ref{thm24})$ when $a_1a_2>1$.
	
	 In the case $(\ref{thm21})$, thanks to its last inequality, we have
	 $$
	 	 \mathfrak v_0
	 	 <-\left (1+\frac{d_2}{r}\right )^{-1}
	 	 \frac{d_1}{a_2}<0.
	 $$
	
	 In the case $(\ref{thm24})$, since $\mathfrak v_0$ increases in $d_2$ and $a_1a_1>1$, we have
	 $$
	 \begin{aligned}
	 	\mathfrak v_0
	 	 <\left (1+\frac{d_2}{r}\right )^{-1}
	 	 \left\{1+\frac{a_1}{a_1a_2-1}\left[4
    		\left (1-\sqrt{1-a_1} +\dfrac{1}{1-\sqrt{1-a_1}}\dfrac{a_1a_2}{a_2+d_1}
    		\right)^{-2}-a_2-d_1\right ]\right\}.
	 \end{aligned}
	 $$
	 Since the right hand side of the above inequality decreases in $d_1$, we obtain that
	 $$
	 \begin{aligned}
	 	\mathfrak v_0
	 	 <&\left (1+\frac{d_2}{r}\right )^{-1}
	 	 \left\{1+\frac{a_1}{a_1a_2-1}\left[4
    		\left (1-\sqrt{1-a_1} +\dfrac{1}{1-\sqrt{1-a_1}}\dfrac{a_1a_2}{a_2+d_1^c}
    		\right)^{-2}-a_2-d_1^c\right ]\right\}\\
    	=&\left (1+\frac{d_2}{r}\right )^{-1}
	\left\{ 1-\frac{a_1}{\left(1-\sqrt{1-a_1}\,\right)^2}
    	\right\}\\
    	<& 0.
	 \end{aligned}
	 $$
	
	 Next, to describe the trajectory of $(\mu,\nu)$, it remains to consider the semi-trivial steady state of $(\ref{uv_f})$. Without loss of generality, we assume that $(\ref{uv_f})$ has two positive semi-trivial steady state
	 $$\left (1-d_1\left(1+\dfrac{d_2}{r}\right)^{-1}\!+\frac{1}{2}\varepsilon,0\right) \quad\text{and} \quad
	 \left (0,\,\left(1+\dfrac{d_2}{r}\right)^{-1}\!-\frac{1}{2}\varepsilon\right ).$$
	 Since it is not difficult to verity that $\nu$ stays above $\mu$ on the $(\bar{\mathfrak u},\underline{\mathfrak v})$ plane when $\varepsilon_0$ is small enough, it follows from the ODE theory of competitive system that
	 $$(\bar{\mathfrak u},\underline{\mathfrak v}) \rightarrow
	 \left (0,\,\left(1+\dfrac{d_2}{r}\right)^{-1}\!-\frac{1}{2}\varepsilon\right )
	 ,\quad t\rightarrow +\infty.$$
	 Hence, there exists $\mathcal T\geq\mathcal T'$ such that for all $t>\mathcal T$, we have $\bar{\mathfrak u}<\frac{1}{4}\varepsilon$ and $\tilde v-\underline{\mathfrak v}<\frac{1}{4}\varepsilon+\frac{1}{2}\varepsilon=\tfrac{3}{4}\varepsilon$, which implies that $(\ref{uvc_epsilon})$ holds for all $t>\mathcal T$.
\end{proof}

Now we are ready to prove Theorem $\ref{thm2}$ by establishing the desired quantitative convergence statement on stabilisation.
\begin{proof}[Proof of Theorem $\ref{thm2}$]
	With Lemma $\ref{wdecay_h}$ and Lemma $\ref{ode_decay_c}$ at hand, by  arguments similar to the method we use to prove $(\ref{w_infi})$, we have $w$ converges exponentially to $\tilde w$ in $L^{\infty}$ as $t\rightarrow\infty$. Therefore, there exist two constants $\mathcal C_6>0$ and $\kappa_2>0$ such that $\gamma$ could be selected as $\gamma(t)=\mathcal C_6e^{-\kappa_2 t}$.
	To obtain the exponential decay rate of $u$ and $v$, we turn back to $(\ref{odec})$. First, by substituting the new $\gamma$ into the first equation of $(\ref{odec})$, we have
	$$
		\begin{aligned}
			\dfrac{\mathrm d}{\mathrm dt}(\bar v_c-\tilde v)
		=& r\bar v_c\big[-(\bar v_c-\tilde v)+\mathcal C_6 d_2e^{-\kappa_2 t}\ \big]\\
		\leq & -r\tilde v(\bar v_c-\tilde v)+r \max \left\{1,\left\|v_0\right\|_{L^{\infty}(\Omega)}\right\} \mathcal C_6 d_2e^{-\kappa_2 t}.
		\end{aligned}
	$$
	By multiplying the above inequality with $e^{\mathcal A_5 t}$, where $\mathcal A_5=\frac{1}{2}\min\{r,\,\kappa_2\}$, it is not difficult to obtain that $\bar v_c$ tends to $\tilde v$ exponentially.
	
	Next, we consider the decay rate of $\bar u_c$. $(\ref{uvc_epsilon})$ implies that there exists $\mathcal T_2>0$ such that for all $t>\mathcal T_2$, we have
	$$
	\begin{aligned}
		\dfrac{\mathrm d}{\mathrm dt}\bar u_c
	=&\bar u_c\Big[-\Big (\frac{a_2+d_1}{1+\frac{d_2}{r}}-1\Big )-\bar u_c+a_2(\tilde v-\underline v_c)+\mathcal C_6 d_1 e^{-\kappa_2 t}\Big]\\
	\leq & -\dfrac{1}{2}\Big (\dfrac{a_2+d_1}{1+\frac{d_2}{r}}-1\Big )\bar u_c.
	\end{aligned}
	$$
	Since the coefficient before $\bar u_c$ is negative due to $(\ref{delta>1})$, we directly have $\bar u_c$ decays exponentially to $ 0$ when $t>\mathcal T_2$.
	
	Finally, we consider the equation of $\underline v_c$ in $(\ref{odec})$. When $t$ is sufficiently large,
	$$
	\begin{aligned}
		\dfrac{\mathrm d}{\mathrm dt}(\tilde v-\underline v_c)
	=&-r\underline v_c(\tilde v-\underline v_c)+r\underline v_c(\mathcal C_6 d_2 e^{-\kappa_2 t}+a_1\bar u_c)\\
	\leq & -\frac{1}{2} r\tilde v(\tilde v-\underline v_c)+r\tilde v(\mathcal C_6 d_2 e^{-\kappa_2 t}+a_1\bar u_c).
	\end{aligned}
	$$
	By the discussion above, the second bracket in the right hand side of the above inequality decays exponentially to $0$ when $t>\mathcal T_2$. Hence, similar to the case of $\bar v_c$, we obtain that $\underline v_c$ tends to $\tilde v$ exponentially when $t$ is sufficiently large.
	
	Summarizing the discussion above, by using Lemma $\ref{ode_clamp_solution_c}$, we have
	$$
	\begin{aligned}
		&||u||_{L^\infty(\Omega)}
		+||v-\tilde v||_{L^\infty(\Omega)}\\
	\leq & ||u||_{L^\infty(\Omega)}
	+||\bar v_c-\tilde v||_{L^\infty(\Omega)}
	+||\underline v_c-\tilde v||_{L^\infty(\Omega)}\rightarrow 0\quad \text{exponentially  as} \quad t\rightarrow\infty.
	\end{aligned}
	$$
	The proof is complete.
	\end{proof}

\section{The healthy state}
This section is devoted to the proof of Theorem $\ref{thm3}$, which is about the global convergence of the  healthy state
     $$(u, v, w) = (1,0,0).
     $$
The main idea of the proof is similar to those of Theorems \ref{thm1} and \ref{thm2}.
The first key step is still the construction of  a proper Lyapunov functional  in the following lemma.

\begin{lemma}\label{uvwr}
    	Suppose that assumptions of Theorem $\ref{thm3}$ hold, $(u,v,w)$ is the global solution of $(\ref{model})$-$(\ref{ini})$. Define
    	$$
    	A_r(t)=\int_\Omega (u(x,t)-1-\ln u(x,t))\, \mathrm dx,
    	$$
    	
    	$$
    	B_r(t)=\int_\Omega v(x,t)\,\mathrm dx,
    	$$
    	
    	$$
    	C_r(t)=\frac{1}{2}\int_\Omega w^2(x,t)\,\mathrm dx.
    	$$
    	Then there exist $\beta_r>0$, $\eta_r>0$ and $\varepsilon_r>0 $ such that the functions $E_r(t)$ and $F_r(t)$ defined by
    	\begin{equation}\label{er}
    		E_r(t)=A_r(t)+\dfrac{\beta_r}{r}B_r(t)+\dfrac{\eta_r}{c}C_r(t), \quad t>0,
    	\end{equation}
    	and
    	\begin{align}\label{fr}
    		F_r(t)=&
    		\int_\Omega (u(x,t)-1)^2\,\mathrm dx
    		+\int_\Omega v(x,t)^2\,\mathrm dx
    		+\int_\Omega w(x,t)^2\,\mathrm dx\nonumber\\
    		&+\int_\Omega \left|\nabla w(x,t)\right|^2\mathrm dx, \quad t>0,
    	\end{align}
    	satisfy
    	\begin{equation}\label{positive_er}
    		E_r(t)\geq 0, \quad t\geq 0,
    	\end{equation}
    	as well as
    	\begin{equation}\label{decay1r}
    		\dfrac{\mathrm d}{{\mathrm dt}}E_r(t)\leq-\varepsilon_r F_r(t).
    	\end{equation}
    \end{lemma}

We will  present the  proof of this lemma in details  at the end since it explains why the conditions on the parameters in Theorem $\ref{thm3}$ are  required.
Thanks to   Lemma \ref{uvwr}, the $L^\infty$ convergence of $w$ to zero is  established as follows.
\begin{lemma}\label{l2con_r}
	Suppose that assumptions of Theorem $\ref{thm3}$ hold, $(u,v,w)$ is the global solution of $(\ref{model})$-$(\ref{ini})$, then
	$$
	||w||_{L^\infty(\Omega)}\rightarrow 0\quad as\quad t\rightarrow\infty.
	$$
\end{lemma}

Again,  Lemma \ref{l2con_r} indicates that there exists
a smooth bounded positive function $\theta(t)$,  which decays to $0$ as $t\rightarrow\infty$ and satisfies
$$
w(x,t) \leq \theta(t),\quad x\in\Omega,\ t \geq 0.
$$
Thus we introduce the following auxiliary ODE system:
\begin{equation}\label{oder}
\begin{cases}
\dfrac{\mathrm d}{\mathrm dt}\underline u_r=\underline u_r\big[1-\underline u_r-a_2\bar v_r-d_1\vartheta(t)\big],&\quad t>0,\\
\dfrac{\mathrm d}{\mathrm dt}\bar v_r=r\bar v_r\big[1-a_1\underline u_r-\bar v_r\big],&\quad t>0,
\end{cases}
\end{equation}
with initial data
\begin{equation}\label{inioder}
\underline u_r(0)=\underline u^r_0:=\min\limits_{\bar\Omega}u_0,\quad
\bar v_r(0)=\bar v^c_0:=\max\limits_{\bar\Omega}v_0,
\end{equation}

Next, under the assumptions of Theorem \ref{thm3},   again parallel to Lemmas \ref{positveode} and \ref{ode_clamp_solution} in Section 3.2   we  derive the follwing  estimates related to the auxiliary ODE system $(\ref{oder})$-$(\ref{inioder})$
\begin{itemize}
	\item  $	0<\underline u_r(t)\leq 1,\  0<\bar v_r(t)\leq 1,  \quad t\geq 0\, ;$
	\item $	\underline u_r(t)\leq u(x,t)<1 , \ \
	0<v(x,t)\leq \bar v_r(t) ,\quad x\in\Omega,\ t \geq 0.$
\end{itemize}
Then since  the auxiliary ODE system $(\ref{oder})$ is also competitive,   similar to proof of Lemma \ref{ode_decay_c},  we have
\begin{equation*}
||u-1||_{L^\infty(\Omega)}
+||v||_{L^\infty(\Omega)}\rightarrow 0
\quad as\quad t\rightarrow\infty.
\end{equation*}
To complete the proof of Theorem \ref{thm3},   the last step is to show
$$
||u-1||_{L^\infty(\Omega)}
+||v||_{L^\infty(\Omega)}
+||w||_{L^\infty(\Omega)}\rightarrow 0\quad \text{exponentially as}\quad t\rightarrow\infty,
$$
by similar arguments in  handling  the homogeneous tumor state at the end of  Section 4.
We omit all the details since they are similar and simpler.

\medskip

It remains to prove Lemma \ref{uvwr}.

\begin{proof}[Proof of Lemma \ref{uvwr}]
	Thanks to the assumption $v_0 \leq 1$, by comparison principle of parabolic equations, we obtain that $v(x,t)\leq1$ for all $x\in \Omega$ and $t>0$. Straightforward computations show
	\begin{align}
		\dfrac{\mathrm d}{\mathrm dt}A_r(t)&=
    	-\int_\Omega (u-1)^2\,\mathrm dx
    	-a_2\int_\Omega(u-1)v\,\mathrm dx
    	-d_1\int_\Omega(u-1)w\,\mathrm dx,\label{dar}\\
    	\frac{1}{r}\dfrac{\mathrm d}{\mathrm dt}B_r(t)
    		&=-a_1\int_\Omega(1-u)v\,\mathrm dx
    		-(a_1-1)\int_\Omega v\,\mathrm dx
    		-\int_\Omega v^2\mathrm dx
    		-\frac{d_2}{r}\int_\Omega vw\,\mathrm dx\nonumber\\
    		&\leq -a_1\int_\Omega(1-u)v\,\mathrm dx
    		-a_1\int_\Omega v^2\,\mathrm dx
    		-\frac{d_2}{r}\int_\Omega vw\,\mathrm dx,\label{dbr}\\
    		\frac{1}{c}\dfrac{\mathrm d}{\mathrm dt}C_r(t)&=
    		-\frac{1}{c}\int_\Omega|\nabla w|^2\,\mathrm dx
    		+\int_\Omega vw\,\mathrm dx
    		-\int_\Omega w^2\,\mathrm dx.\label{dcr}
	\end{align}
    By differentiating $(\ref{er})$ and substituting $(\ref{dar})$-$(\ref{dcr})$ into it, we obtain

    \begin{equation}\label{er_differentiate}
    		\dfrac{\mathrm d}{\mathrm dt}E_r(t)
    			\leq -\int_\Omega\mathbf Z^{\mathrm T}\mathbb P_r \mathbf Z\,\mathrm dx
    			-\dfrac{\eta_r}{c}\int_\Omega\left| \nabla w \right|^2\mathrm dx,
    	\end{equation}
    	where
    	\begin{equation}
    		\mathbb P_r=
    		\begin{pmatrix}\displaystyle
    			1 & \dfrac{a_2+a_1\beta_r}{2} & \dfrac{d_1}{2}\\
    			\dfrac{a_2+a_1\beta_r}{2} & a_1\beta_r & \dfrac{\dfrac{d_2}{r}\beta_r-
    			\eta_r}{2}\\
    			\dfrac{d_1}{2} & \dfrac{\dfrac{d_2}{r}\beta_r-
    			\eta_r}{2} & \eta_r
    		\end{pmatrix}
    	\end{equation}
    	and
    	$$
    	\mathbf Z=\left(u-1,v,w\right)^\mathbf T.
    	$$
    	
    Similar to the previous two cases, we claim that there exist two positive constants $\beta_r$ and $\eta_r$ such that $\mathbb P_r$ is positive definite if and only if there exists $\beta_r>0$ satisfies the following two inequalities simultaneously:
    \begin{numcases}{}
    		\Phi_r(\beta_r):=
    		-a_1^2\beta_r^2+
    		2\Big [\,2\Big(a_1+\dfrac{d_2}{r}\Big)-a_1(a_2+d_1)\,\Big  ]\beta_r-(a_2+d_1)^2>0,\label{group1r}\\
    		\Psi_r(\beta_r):=
    		-a_1^2\beta_r^2+2(2a_1 -a_1a_2)\beta_r-a_2^2>0\label{group2r}.
   	\end{numcases}
   	
   	For simplicity, we denote $\alpha:=\tfrac{1}{2}(a_2+a_1\beta_r)$. To verify our assertion, we just need to compute all the principal minors of $\mathbb P_r$:
   	$$
    	\mathbf {M^r_1}:=1,
    	$$
    	$$
    	\mathbf{M^r_2}:=
    	\begin{vmatrix}
    		1 & \alpha\\
    		\alpha & a_1\beta_r
    	\end{vmatrix}
    	=a_1\beta_r-\alpha^2
    	=\dfrac{1}{4}\left(
    	-a_1^2\beta_r^2+2(2a_1 -a_1a_2)\beta_r-a_2^2\,
    	\right)=\dfrac{1}{4}\Psi_r(\beta_r),
    	$$
   and $(\ref{group2r})$ holds if and only if $\mathbf{M^r_2}>0$. Now, we consider the discriminant of $\mathbb P_r$.
   	\begin{equation}\label{quadratic-r}
   		\begin{aligned}
   			&\det{\mathbb P_r}
   		=\frac{1}{4}
   		\left\{
   		-\eta_r^2
   		+2\Big(2(a_1\beta_r-\alpha^2)+\Big(\dfrac{d_2}{r}\beta_r-\alpha d_1\Big)\Big)\eta_r\right.\\
    	&\qquad\qquad\quad\left.+\Big(2\alpha d_1\dfrac{d_2}{r}\beta_r-\Big(\dfrac{d_2}{r}\beta_r\Big)^2-a_1d_1^2\beta_r\Big)
   		\right\}.
   		\end{aligned}
   	\end{equation}
   	Notice from $(\ref{group2r})$ and $a_1>1$ that
    	$$
    	2\alpha d_1\dfrac{d_2}{r}\beta_r-\Big(\dfrac{d_2}{r}\beta_r\Big)^2-a_1d_1^2\beta_r
    	<-\beta_r\Big(\dfrac{d_2}{r}\alpha-d_1\Big)^2 \leq 0,
    	$$
    	fundamental properties of quadratic polynomials implies that there exists $\eta_r>0$ such that $\det{\mathbb P_r}>0$ if and only if the following situation holds:
    	\begin{equation}\label{stronger_r}
    		\begin{cases}
    			{\Delta_r}>0,\\
    			2(a_1\beta_r-\alpha^2)+\Big(\dfrac{d_2}{r}\beta_r-\alpha d_1\Big)\geq 0,
    		\end{cases}
    	\end{equation}
   	where ${\Delta_r}$ is the discriminant of the quadratic $(\ref{quadratic-r})$. By calculating this discriminant and substituting $\alpha=\tfrac{1}{2}(a_2+a_1\beta_r)$ into it, we have
   	$$
   	\Delta_r
   	=16\Big[(a_1+d_2)\beta_r-\alpha d_1-\frac{1}{4}d_1^2\,\Big]
   	(a_1\beta_r-\alpha^2)
   	=\Phi_r(\beta_r)\Psi_r(\beta_r)
   	$$
   	Since we already have $(\ref{group2r})$, it follows from the above equations that ${\Delta_r}>0$ if and only if $(\ref{group1r})$ holds.
    	Also, when $\Delta_r>0$ and $\Psi_r(\beta_r)>0$,
    	$$
    	2(a_1\beta_r-\alpha^2)+\Big(\dfrac{d_2}{r}\beta_r-\alpha d_1\Big)
    	>\Big(a_1+\frac{d_2}{r}\Big)\beta_c-\alpha^2 -\alpha d_1
    	>\dfrac{1}{4}\Phi_c(\beta_c)>0,
    	$$
    	which implies the second equation in $(\ref{stronger_r})$ is automatically satisfied.
    	Therefore, on the basis of $(\ref{group2r})$, there exist two positive constants $\beta_r$ and $\eta_r$ such that $\det{\mathbb P_r}>0$ if and only if $(\ref{group1r})$ holds. Summing up the discussion above, our assertion has been proved.
   	
   	Now, it remains to show that under the assumptions of Theorem $\ref{thm3}$, there exists $\beta_r>0$ which satisfies $(\ref{group1r})$ and $(\ref{group2r})$ simultaneously. For this purpose, we denote the positive solution of $(\ref{group1r})$ as $S^r_1:=\left((L^r_1)^2,\, (R^r_1)^2\right)$ and positive solution of $(\ref{group2r})$ as $S^r_2:=\left( (L^r_2)^2,\,(R^r_2)^2\right)$.
    	We assume for now that we have
    	\begin{equation}\label{d2-new-r}
    		\dfrac{d_2}{r} > a_1(a_2+d_1-1),
    	\end{equation}
    	which is already contained in the case $(\ref{d2_r_1})$.
    	Thanks to $(\ref{d2-new-r})$, $S^r_1$ is not empty. On the other hand, $S^r_2$ is not empty due to $a_2<1$.
    	Then, direct computations show that
    	$$
    	\begin{aligned}
    	&L^r_1=\dfrac{1}{a_1}\left (\sqrt{a_1+\frac{d_2}{r}}-\sqrt{a_1+\frac{d_2}{r}-a_1(a_2+d_1)}\,\right ) ,\\
    	&R^r_1=\dfrac{1}{a_1}\left (\sqrt{a_1+\frac{d_2}{r}}+\sqrt{a_1+\frac{d_2}{r}-a_1(a_2+d_1)}\,\right ),\\
    	&L^r_2=\dfrac{1}{\sqrt{a_1}}\big (1-\sqrt{1-a_2}\,\big ) ,\quad
    	R^r_2=\dfrac{1}{\sqrt{a_1}}\big (1+\sqrt{1-a_2}\,\big ).
    	\end{aligned}
    	$$
    	Since $\frac{d_2}{r}>a_1(a_2+d_1-1)$, we have
    	$$
    	 R^r_1
    	 > \dfrac{\sqrt{a_2+d_1+1}}{\sqrt{a_1}}
    	 > \dfrac{1}{\sqrt{a_1}}
    	 > L^r_2,
    	$$
    	we only need $L^r_1 < R^r_2$, namely
    	\begin{equation}\label{inequality_r}
    		\sqrt{a_1+\frac{d_2}{r}}-\sqrt{a_1+\frac{d_2}{r}-a_1(a_2+d_1)}
    		< \sqrt{a_1}
    		\big (1+\sqrt{1-a_2}\,\big ),
    	\end{equation}
    	so that there is overlap part between $S^r_1$ and $S^r_2$.
   	Recall that we also need $(\ref{d2-new-r})$, hence, in the following part, we verify that $(\ref{d2-new-r})$ and $(\ref{inequality_r})$ hold under the assumption of Theorem $\ref{thm3}$.
   	
   	First, we consider the case $(\ref{d2_r_1})$. Since in this case we already have  $(\ref{d2-new-r})$, it remains to show that when $d_1\leq d_1^r$, where $d^h_1$ is defined in $(\ref{d1r})$, we can derive $(\ref{inequality_r})$ from $(\ref{d2-new-r})$. By numerator rationalization of $(\ref{inequality_r})$, we obtain
   	\begin{equation}\label{rationed-r}
   		\frac{a_1(a_2+d_1)}{1+\sqrt{1-a_2}}
   		<\sqrt{a_1+\frac{d_2}{r}}+\sqrt{a_1+\frac{d_2}{r}-a_1(a_2+d_1)}.
   	\end{equation}
   	Substituting $\frac{d_2}{r}=a_1(a_2+d_1-1)$ into the inequality above yields
    	$$
    	\sqrt{a_2+d_1}<1+\sqrt{1-a_2} ,
    	$$
    	which is equivalent to $d_1 < d^r_1$. Since $\frac{d_2}{r}$ is  strictly smaller than $a_1(a_2+d_1-1)$, we obtain that $(\ref{rationed-r})$ still holds when $d_1=d_1^r$.
    	Notice that the right hand side of $(\ref{rationed-r})$ increases in $d_2$, we obtain that $(\ref{inequality_r})$ still holds when $(\ref{d2-new-r})$ is satisfied.
   	
   	Next, we demonstrate that in the case $(\ref{d2_r_2})$, we have $(\ref{d2-new-r})$ and $(\ref{inequality_r})$. Direct computations show that $(\ref{inequality_r})$ is equivalent to $\frac{d_2}{r}>d_2^r$, where $d_2^r$ is defined in $(\ref{d2r})$. Hence, we have $(\ref{inequality_r})$ in the case $(\ref{d2_r_2})$. On the other hand, it follows from the discussion in $d_1\leq d_1^r$ part that when $d_1>d_1^h$, there is $d_2^h>a_2+d_1-1$. Hence, in the case $(\ref{d2_r_2})$ we have $(\ref{d2-new-r})$.
   	
   	Summarizing the discussion above, we draw out that  assumptions of Theorem $\ref{thm3}$ suffice to show the existence of positive $\beta_r$ and $\eta_r$ such that $\mathbb P_r$ is positive definite. By the definition of positive definite matrix, there exists a constant $\varepsilon_3>0$ such that
        $$
    	\mathbf Z^{\mathrm T}\mathbb P_r \mathbf Z\geq
    	\varepsilon_3 |\mathbf Z|^2.
    	$$
    	Substituting it into $(\ref{er_differentiate})$, we have
    	\begin{equation}\label{erinquality}
    		\dfrac{\mathrm d}{\mathrm dt}E_r(t)
    		\leq -\varepsilon_3\int_\Omega|\mathbf Z|^2\,\mathrm dx
    		-\dfrac{\eta_r}{c}\int_\Omega\left| \nabla w \right|^2\mathrm dx
    	\leq -\varepsilon_r F_r(t),
    	\end{equation}
    	where $\varepsilon_r=\min\{ \varepsilon_3,\frac{\eta_r}{c} \}$.	
\end{proof}

\end{document}